\documentclass[11pt]{amsart}

 \usepackage[utf8]{inputenc}

 \usepackage[english]{babel}

 \usepackage{amsmath,amsfonts,amssymb}
 \usepackage{mathrsfs}
 \usepackage{amsmath}
 \usepackage{dsfont} 
 \usepackage{mathtools} 
 \usepackage{esint} 
 \usepackage{xfrac} 
 \usepackage{defs} 
 \usepackage{fonts} 
 \usepackage[normalem]{ulem}

 \usepackage{amsthm}
 \usepackage[shortlabels]{enumitem} 
 \usepackage[colorlinks=true,urlcolor=blue, citecolor=red,linkcolor=blue,linktocpage,pdfpagelabels, bookmarksnumbered,bookmarksopen]{hyperref} 
 \usepackage[paper=a4paper,hmargin=1in,vmargin=1in, marginparwidth=0.8in]{geometry} 

 \usepackage[colorinlistoftodos,prependcaption,linecolor=blue,backgroundcolor=blue!25,bordercolor=blue,textsize=tiny]{todonotes}
 \usepackage{natbib}
 \usepackage{cleveref}
 \usepackage{comment}

plus 6pt minus 12pt
\numberwithin{equation}{section}

\theoremstyle{plain}
\newtheorem{theorem}{Theorem}[section]
\newtheorem{proposition}[theorem]{Proposition}
\newtheorem{lemma}[theorem]{Lemma}
\newtheorem{corollary}[theorem]{Corollary}
\theoremstyle{definition}

\newtheorem{example}{Example}[section]

\theoremstyle{remark}
\newtheorem{remark}[theorem]{Remark}

\title[Connections between nonlocal operators]{
Connections between nonlocal operators: \\ from vector calculus identities to a fractional Helmholtz decomposition}
\author{Marta D'Elia}
\address[Marta D'Elia]{Data Science and Computing group, Sandia National Laboratories, CA}
\email{mdelia@sandia.gov}

\author{Mamikon Gulian}
\address[Mamikon Gulian]{Data Science and Computing group, Sandia National Laboratories, CA}
\email{mgulian@sandia.gov}

\author{Tadele Mengesha}
\address[Tadele Mengesha]{Department of Mathematics,
University of Tennessee,
Knoxville, TN}
\email{mengesha@utk.edu}

\author{James M. Scott}
\address[James M. Scott]{Applied Physics and Applied Mathematics,
Columbia University, NY}
\email{jms2555@columbia.edu}

\newcommand{\xb}{\mathbf{x}}
\newcommand{\yb}{\mathbf{y}}
\newcommand{\zb}{\mathbf{z}}
\newcommand{\ub}{\mathbf{u}}

\newcommand{\vb}{\mathbf{v}}

\newcommand{\ab}{\mathbf{a}}
\newcommand{\cb}{\mathbf{c}}
\newcommand{\hb}{\mathbf{h}}

\begin{document}

\maketitle

\begin{abstract}
Nonlocal vector calculus, which is based on the nonlocal forms of gradient, divergence, and Laplace operators in multiple dimensions, has shown promising applications in fields such as hydrology, mechanics, and image processing. 
In this work, we study the analytical underpinnings of these operators. We rigorously treat compositions of nonlocal operators, prove nonlocal vector calculus identities, and connect weighted and unweighted variational frameworks. We combine these results to obtain a weighted fractional Helmholtz decomposition which is valid for sufficiently smooth vector fields. Our approach identifies the function spaces in which the stated identities and decompositions hold, providing a rigorous foundation to the nonlocal vector calculus identities that can serve as tools for nonlocal modeling in higher dimensions.\end{abstract}

\section{Introduction}

Nonlocal operators are operators whose functional values are determined by integration over a neighborhood, in contrast to differential operators which are locally determined. The integral nature of these operators allows them to describe multiscale behavior and anomalous behavior such as super- and sub-diffusion. This feature makes nonlocal models a viable alternative to models based on partial differential equations (PDEs) for a broad class of engineering and scientific applications. Such applications include in groundwater hydrology for subsurface transport \citep{Benson2000,d2021analysis,Schumer2003,Schumer2001}, image processing \citep{Buades2010,Gilboa2007,DElia2021Imaging}, multiscale and multiphysics systems \citep{Alali2012,Du-Lipton-Mengesha,Askari2008},
finance \citep{Scalas2000,Sabatelli2002}, and stochastic processes \citep{Burch2014,DElia2017,Meerschaert2012,Metzler2000,Metzler2004}.

The foundations of nonlocal vector calculus, based on the nonlocal gradient, divergence, and Laplace operators in multiple dimensions, were developed by \citet{Gunzburger2010}, \citet{Du2012}, \citet{du2013analysis}, and \cite{Du2013}. In these works, two frameworks were introduced: an \emph{unweighted} framework and a \emph{weighted} framework. The unweighted framework involves the two-point gradient $\cG \ub(\xb,\yb)$ and its adjoint, the nonlocal divergence operator $\cD \ub(\xb)$, the composition of which yields a nonlocal Laplace operator.
The weighted framework is based on the one-point weighted gradient $\cG_\varrho \ub(\xb)$ and its adjoint $\cD_\varrho \ub(\xb)$, the weighted divergence. The one-point structure characterizing these weighted operators makes them more amenable for certain applications; see \cite{Du2018Dirichlet,lee2020nonlocal} for applications to mechanics and \cite{Du2020SPH} for an application to fluid dynamics. Rigorous analysis of important aspects pertaining to these operators was performed by \cite{Mengesha-Spector} and \citet{MENGESHA201682}. 

Various forms of fractional-order (hereafter referred to as \emph{fractional}) vector calculus have been developed both independently and in parallel; see for instance \citet{Meerschaert2006}, \cite{silhavy2020fractional}, and \citet{Tarasov2008}. \citet{d2020unified} showed that a widely used form of fractional-vector calculus is in fact a special case of the weighted nonlocal vector calculus with singular weight function $\varrho$ and infinite interaction radius. In particular, it was noted that the fractional gradient and divergence are special cases of weighted nonlocal operators. Moreover, despite the fractional Laplacian having an immediate representation as $\mathcal{D} \circ \mathcal{G}$, a composition of unweighted operators, it was also shown to be represented as $\mathcal{D}_\varrho \circ \mathcal{G}_\varrho$, a composition of the weighted fractional divergence and gradient. 
This representation of weighted Laplace operators as unweighted diffusion operators was formally extended to the more general kernel-based  nonlocal calculus in \cite{d2020unified} by deriving an associated {\it equivalence kernel.}
This result reinforces ideas discovered in prior studies on equivalence kernels by \citet{Alali-preprint,mengesha2014peridynamic}, and \citet{vsilhavy2017higher}  in the context of peridynamics, a nonlocal model of mechanics in which the nonlocal Navier-Lam\'e operator is represented as a nonlocal Laplace-type operator. 

The aforementioned representations clearly show the major role that composition of operators play in deriving useful nonlocal vector calculus identities. 
A major contribution of this work is the rigorous justification of various representations of compositions of weighted nonlocal operators.
We provide conditions under which the composition of two nonlocal operators defined by principal value integrals, such as the fractional divergence and gradient, can be represented by a double principal value integral. These analytical results are utilized together with classical vector calculus identities to prove several identities for weighted nonlocal vector operators, such as 
\begin{equation*}
\cC_\omega \circ \cG_\omega = 0, \quad 
\cD_\omega \circ \cC_\omega = 0, \quad 
\cC_\omega \circ \cC_\omega = 
\cG_\omega \circ \cD_\omega 
- \cD_\omega \circ \cG_\omega
\end{equation*}
for translation invariant kernels, including fractional kernels.  We specify the space of functions over which such composition is possible. 

Another contribution is a \emph{rigorous} proof of the equivalence of weighted and unweighted nonlocal Laplace operators via the equivalence kernel. 
While this result was presented in \citet{d2020unified} formally, here we  provide a set of conditions under which the result is valid. We verify these conditions for several important classes of kernels, including fractional kernels. We further study the properties of the equivalence kernel, which are important for establishing well-posedness for weighted nonlocal models. 

Finally, we combine our results to obtain a weighted fractional Helmholtz decomposition in H\"older spaces. This result utilizes the vector calculus identities proved in the first half of the paper, as well as the characterization of the equivalence kernel for fractional kernels. A nonlocal Helmholtz decomposition for unweighted operators was derived by \cite{DElia2020Helmholtz}.
For weighted nonlocal operators, a Helmholtz decomposition for operators with kernels supported in the half ball was derived by \citet{lee2020nonlocal}. Their results bear resemblance to the Helmholtz decompositions derived in the present paper,  but in a different setting, namely in a periodic domain and for nonlocal kernels that scale in certain limits to local operators. In contrast, we study such decompositions in $\mathbb{R}^d$ with relaxed assumptions about the decay at infinity, which hold for standard fractional operators. 
In another related work, \citet{petronela_poster} studied Helmholtz decompositions for nonlocal convolution operators. 

\noindent{\it Outline of the paper.} 
In Section \ref{sec:OperatorDef}, {we introduce our notation and recall several relevant results for nonlocal operators in multiple dimensions with well-known kernels. We establish basic mapping properties of these operators as well. In Section \ref{sec:holder}, we focus on fractional operators and characterize their mapping properties completely for several function spaces.} In Section \ref{sec:identities}, we prove several nonlocal operator identities that reflect well-established local counterparts from classical vector calculus. In Section \ref{sec:eq-kernel}, we identify a specific class of functions for which there exists an equivalence kernel such that the composition of the divergence and gradient operators corresponds to the (unweighted) negative nonlocal Laplace operator. Finally, in Section \ref{sec:helmholtz} we combine the vector calculus identities and the characterization of the equivalence kernel to obtain a weighted fractional Helmholtz decomposition. We collect requisite properties of the hypergeometric function in Appendix \ref{apdx:HyperGeometricFxn}.

\section{Definitions of Operators}\label{sec:OperatorDef}

In this section, we recall the definitions of the nonlocal operators that will be used throughout the paper and identify function spaces for which these operators are defined. Furthermore, we introduce examples of nonlocal kernel functions that will be utilized in our main results. 

Suppose we have a radial kernel $\varrho$ satisfying 
\begin{equation}\label{assumption:Kernel}
    \begin{gathered}
        \varrho \in L^1_{\text{loc}}(\bbR^d)\,, \qquad \varrho \geq 0\,, \qquad \frac{\varrho(\bseta)}{|\bseta|} \text{ is nonincreasing in } |\bseta|\,, \\
        \text{ and } \intdm{\bbR^d}{ \min \{1,|\bseta|^{-1} \} \varrho(\bseta) }{\bseta} < \infty\,.
        \tag{K}
    \end{gathered}
\end{equation}
Given positive integers $N, d$ and a given vector field $\ub: \mathbb{R}^d \rightarrow \mathbb{R}^N$ we define the nonlocal gradient 
\begin{align}\label{eq:definitions_of_ops}
\cG_{\varrho}\ub(\xb) &= \int_{\bbR^d} \varrho (\yb-\xb)  \frac{\ub(\yb)-\ub(\xb)}{|\by-\bx|} \otimes \frac{\by-\bx}{|\by-\bx|}  \, \rmd\yb 
\end{align}
where for any vector ${\bf a}\in \bbR^N$ and ${\bf b}\in \bbR^d$, the tensor product ${\bf a}\otimes {\bfb}$ is the $N \times d$ matrix the $ij^{\text{th}}$ entry of which is $a_ib_j$. From this definition, it is clear that $\cG_{\varrho}$ is an $N \times d$ matrix-valued map. 
For a given second order tensor field $\ub : \mathbb{R}^d \rightarrow \mathbb{R}^{N \times d}$, we define the nonlocal divergence as 
\begin{equation}\label{eq:div-tensor}
\cD_{\varrho}\ub(\xb) = \int_{\bbR^d} \varrho (\yb-\xb) \frac{\ub(\yb)-\ub(\xb)}{|\by-\bx|} \frac{\by-\bx}{|\by-\bx|} \, \rmd \yb.
\end{equation}
We interpret $\mathbb{R}^{N \times d}$  as a set of matrices with $N$ rows and $d$ columns.
It is clear from the above definition \eqref{eq:div-tensor} that $\cD_{\varrho} \bu(\bx)$ is $\bbR^N$-valued.
Finally, in the event $d=N=3$ and $\ub:\bbR^3\to\bbR^3$ we define the nonlocal curl operator as 
\begin{align}\label{eq:curl-vector}
\cC_{\varrho}\ub(\xb) = \int_{\bbR^3} \varrho (\yb-\xb) \frac{\by-\bx}{|\by-\bx|} \times \frac{\ub(\yb)-\ub(\xb)}{|\by-\bx|} \, \rmd\yb.
\end{align}
The definitions above are consistent with known corresponding definitions for scalar fields or vector fields given in, e.g., \cite{Du2013}. Indeed, in the event that $u:\bbR^d\to \bbR$ is a scalar field, the nonlocal gradient operator acting on $u$ gives the vector field $\cG_{\varrho} u$ where 
\begin{equation*}
\cG_{\varrho} u(\xb) = \int_{\bbR^d} \varrho (\yb-\xb)  \frac{u(\yb)-u(\xb)}{|\by-\bx|} \frac{\by-\bx} {|\by-\bx|}  \, \rmd\yb.
\end{equation*}
If we identify vector fields $\ub: \bbR^d\to \bbR^d$ as $1\times d$ matrix-valued fields, then the nonlocal divergence operator acting on the vector field is the scalar function given by 
\begin{align*}
\cD_{\varrho}\ub(\xb) &= \int_{\bbR^d} \varrho (\yb-\xb) \frac{\by-\bx}{|\by-\bx|} \cdot \frac{\ub(\yb)-\ub(\xb)}{|\by-\bx|} \, \rmd\yb. 
\end{align*}
{Note that in the literature on nonlocal vector calculus (see, e.g, \cite{Du2013}) these operators are commonly referred to as {\it weighted}, as opposed to their unweighted counterparts.} 

{
In what follows, we denote the $i$-th partial derivative by $D_i$, and we use multi-index notation for derivatives of arbitrary order. The (classical) gradient operator $(D_1, D_2, \ldots, D_n)$ is denoted by $\grad$, its negative adjoint, the (classical) divergence operator, is denoted by $\div$, and their composition, the (classical) Laplacian $\div \grad$, is denoted by $\Delta$.

For functions $\bu : \bbR^d \to \bbR^N$, for $k \in \bbN_0$ and $\alpha \in (0,1)$ we denote the H\"older spaces as $C^{k,\alpha}(\bbR^d;\bbR^N) := \{ \bu \in C(\bbR^d;\bbR^N) \, : \, \Vnorm{\bu}_{C^{k,\alpha}(\bbR^d)} < \infty \}$, where the norm is given by
\begin{gather*}
\| \ub \|_{C^{k,\alpha}(\mathbb{R}^d)} 
= \Vnorm{\bu}_{L^{\infty}(\bbR^d)} +
\sum_{|\gamma|=1}^k \Vnorm{ D^{\gamma} \ub }_{L^\infty(\mathbb{R}^d)}
+
\sum_{|\gamma|=k} [ D^{\gamma} \ub ]_{C^{0,\alpha}(\mathbb{R}^d)}, \\
\left[ \vb \right]_{C^{0,\alpha}(\mathbb{R}^d)}
=
\underset{\xb,\yb \in \mathbb{R}^d, \xb \neq \yb}{\text{sup}}
\frac{|\vb(\xb) - \vb(\yb)|}{|\xb-\yb|^{\alpha}}.
\end{gather*}

Later on in the paper we will need to consider functions in a range of H\"older spaces that depend on a parameter $s$. For all $s \in (0,1)$ and $\sigma > 0$ small, we say that 
\begin{equation*}
    \bu \in \scC^{2s+\sigma}(\bbR^d;\bbR^N)
\end{equation*}
if
\begin{equation*}
    \bu \in 
    \begin{cases}
    C^{0,2s+\sigma}(\bbR^d;\bbR^N)\,, &\quad \text{ when } s < 1/2\,, \\
        C^{1,2s+\sigma-1}(\bbR^d;\bbR^N)\,, &\quad \text{ when } s \geq 1/2\,.
    \end{cases}
\end{equation*}
}

We next study the mapping properties of the nonlocal operators introduced above. To that end,  we recall the class of Schwartz vector fields by $\scS(\bbR^d;\bbR^N)$: this is the space $C^{\infty}(\bbR^d;\bbR^N)$ equipped with the countable family of seminorms
\begin{equation*}
    [\bu]_{\alpha,\beta} := \sup_{|\gamma|\leq \alpha} \sup_{\bx \in \bbR^d} |\bx|^{\beta} |D^{\gamma}\bu(\bx)|\,, \qquad \alpha, \beta \in \bbN_0\,, \quad \gamma \text{ a } d\text{-multi-index.}
\end{equation*}

Our next result says that although these nonlocal operators do not necessarily map $\scS(\bbR^d;\bbR^N)$ to itself, they map it to the class $C^{\infty}$, and the mapped vector fields satisfy a certain decay property. 

\begin{proposition}\label{decay-estimates-operator}
Let $d$ and $N$ be positive integers. Let $\cZ_{\varrho} \bu(\bx)$ denote any one of the following objects:
\begin{equation}\label{eq:OperatorAbbreviations}
    \begin{split}
        \cG_{\varrho} \bu(\bx)\,, &\quad \text{ for } \bu \in \scS(\bbR^d;\bbR^N)\,,  \\
        \cD_{\varrho} \bu(\bx)\,, &\quad \text{ for } \bu \in \scS(\bbR^d;\bbR^{N \times d})\,,  \\
        \cC_{\varrho} \bu(\bx)\,, &\quad \text{ for } \bu \in \scS(\bbR^d;\bbR^d) \text{ and } d=3\,.
    \end{split}
\end{equation}
Then $\cZ_{\varrho} \bu(\bx)$ is a well-defined measurable function for all $\bx$, $\cZ_{\varrho} \bu \in C^{\infty}$, and for any $p\in [1, \infty]$ and $\gamma \in \bbN^d$, there is a constant $C$ depending on $d$, $N$ and $p$ such that  
\begin{equation}\label{eq:Lp-NonlocalOperator}
\|D^{\gamma}\cZ_{\varrho}\bu\|_{L^{p}(\bbR^d)} \leq C\left(\|\grad D^{\gamma}\bu\|_{L^{p}(\bbR^d)}\|\varrho\|_{L^{1}(B_1(0))} + \|D^{\gamma}\bu\|_{L^{p}(\bbR^d)}\left\|{\varrho(\cdot)\over |\cdot|}\right\|_{L^1 (\bbR^d\setminus B_1(0))}\right).
\end{equation}
Moreover, we have the following decay estimates for the derivatives: for any $j$, $k \in \bbN$ , there exists a constant $C$ depending on $d$, $N$, $j$ and $k$ such that 
\begin{equation}\label{eq:DecayRate}
    \begin{split}
        |D^{\gamma}\cZ_{\varrho}\bu(\bx)| \leq C \left( \frac{[\bu]_{|\gamma|+1,j}}{|\bx|^j} \intdm{|\bh| \leq \frac{|\bx|}{2} }{ \varrho(|\bh|)}{\bh} + \frac{[\bu]_{|\gamma|,k}}{|\bx|^k} \intdm{|\bh| > \frac{|\bx|}{2} }{ \frac{\varrho(|\bh|)}{|\bh|}  }{\bh} + \|D^{\gamma}\bu\|_{L^{1}(\bbR^d)}\frac{ \varrho\left( \frac{|\bx|}{2} \right)}{|\frac{\bx}{2}|}   \right)
    \end{split}
\end{equation}
for all $|\bx| \geq 1$.
\end{proposition}

\begin{proof}
First, we show that for $\bu \in \scS(\bbR^d)$, $\cZ_{\varrho} \bu(\bx)$ is well defined for any fixed $\bx \in \bbR^d$. From the definition of these operators and after change of variables, we notice that 
\begin{equation*}
   |\cZ_{\varrho} \bu(\bx)|\leq \intdm{\bbR^d}{\varrho(|\bz|)\frac{|\bu(\bx+\zb)-\bu(\bx)|}{|\zb|}}{\zb}.
\end{equation*}
Thus to show that $\cZ_{\varrho} \bu(\bx)$ is well defined, it suffices to show that the integrand on the right-hand side is integrable. This follows from the fact that the integrand can be estimated by a sum of two integrable functions, i.e. 
\begin{equation}\label{eq:L1Est1}
\varrho(|\bz|)\frac{|\bu(\bx+\zb)-\bu(\bx)|}{|\zb|} \leq \|\grad \bu\|_{L^{\infty}(B_1(\bx))} \varrho(\zb)\chi_{\{|\zb|\leq 1\}}(\zb) + 
2 \|\ub\|_{L^\infty(\mathbb{R}^d)} {\varrho(\zb)\over |\zb|}\chi_{\{|\zb|\geq 1\}}(\zb).
\end{equation}

The fact that $\cZ_{\varrho} \bu \in C^{\infty}(\bbR^d;\bbR^d)$ follows from the observation that the operators $\cZ_{\varrho}$ commute with differentiation for vector fields in $\scS(\bbR^d)$, i.e. for any multi-index $\gamma\in \bbN_0^d$, $D^\gamma \cZ_{\varrho} \bu = \cZ_{\varrho} D^{\gamma}\bu$. This commutative property follows by induction from the relation $D_i \cZ_{\varrho} \bu = \cZ_{\varrho} D_i \bu$ for $i = 1, \ldots, d$. This, in turn, can be seen by applying the estimates \eqref{eq:L1Est1} and 
\begin{multline}\label{eq:L1Est2}
    \varrho(|\bz|)\frac{|D_i \bu(\bx+\zb)- D_i \bu(\bx)|}{|\zb|}
    \\
    \leq \|\grad D_i \bu\|_{L^{\infty}(B_1(\bx))} \varrho(\zb)\chi_{\{|\zb|\leq 1\}}(\zb) + 
2 \|D_i \ub\|_{L^\infty(\mathbb{R}^d)} {\varrho(\zb)\over |\zb|}\chi_{\{|\zb|\geq 1\}}(\zb)
\end{multline}
in the dominated convergence theorem in order to differentiate under the integral sign. Thus to demonstrate the estimates \eqref{eq:Lp-NonlocalOperator}
and  \eqref{eq:DecayRate}, it suffices to check it for $\gamma =0$. To prove \eqref{eq:Lp-NonlocalOperator} when $1\leq p<\infty$, we have
\begin{multline}
\int_{\bbR^d}|\cZ_{\varrho}\bu|^p \,  \rmd \bx \leq C \int_{\bbR^d}\left|\int_{B_1(0)}\varrho(|\bz|)\frac{|\bu(\bx+\zb)-\bu(\bx)|}{|\zb|} \, \rmd \zb\right|^p \, \rmd \bx\\
+ C \int_{\bbR^d}\left|\int_{\bbR^{d}\setminus B_1(0) }\varrho(|\bz|)\frac{|\bu(\bx+\zb)-\bu(\bx)|}{|\zb|} \, \rmd \zb\right|^p \, \rmd\bx\,.    
\end{multline}
Using the identity $\bu(\bx+\zb)-\bu(\bx) = \int_{0}^{1}\grad \bu(\bx +t\zb)\zb \, \rmd t$ and applying Minkowski's integral inequality, we have that 
\begin{align*}
\int_{\bbR^d}|\cZ_{\varrho}\bu|^p \, \rmd \bx &\leq C \int_{\bbR^d}\left|\int_{B_1(0)}\varrho(|\bz|)\int_{0}^{1}|\grad \bu(\bx +t\zb)| \, \rmd t \, \rmd \zb\right|^p \, \rmd \bx\\
&\qquad +\int_{\bbR^d}\left|\int_{\bbR^{d}\setminus B_1(0) }\varrho(|\bz|)\frac{|\bu(\bx+\zb)-\bu(\bx)|}{|\zb|} \, \rmd\zb\right|^p \, \rmd\bx\\
&\leq C\left[\left(\int_{B_1(0)}\varrho(|\zb|) \, \rmd \zb\right)^p \|\grad \bu\|_{L^{p}(\bbR^d)} +\left( \int_{\bbR^{d}\setminus B_1(0) }{\varrho(|\bz|) \over |\zb|} \, \rmd \zb \right)^p \|\bu\|_{L^{p}(\bbR^d)}\right]\,.
\end{align*}
The estimate \eqref{eq:Lp-NonlocalOperator} in the case $p=\infty$ is similar:
\begin{align*}
\Vnorm{\cZ_{\varrho}\bu}_{L^{\infty}(\bbR^d)}
&\leq C\left[\left(\int_{B_1(0)}\varrho(|\zb|) \, \rmd \zb\right) \|\grad \bu\|_{L^{\infty}(\bbR^d)} +\left( \int_{\bbR^{d}\setminus B_1(0) }{\varrho(|\bz|) \over |\zb|} \, \rmd \zb \right) \|\bu\|_{L^{\infty}(\bbR^d)}\right]\,.
\end{align*}

To estimate \eqref{eq:DecayRate}, we proceed by splitting the integral defining $\cZ_{\varrho} \bu$:
\begin{equation*}
    \begin{split}
    |\cZ_{\varrho} \bu(\bx)| &\leq \intdm{|\zb|\leq \frac{|\bx|}{2} }{ \varrho(|\zb|) \frac{|\bu(\zb + \bx)-\bu(\bx)|}{|\zb|} }{\zb} + \intdm{|\zb| > \frac{|\bx|}{2} }{ \varrho(|\zb|) \frac{|\bu(\zb+\bx)-\bu(\bx)|}{|\zb|} }{\zb} \\
    &\leq \intdm{|\zb|\leq \frac{|\bx|}{2} }{ \varrho(|\zb|) \grad \bu(\bx +t(\zb) ) }{\by} + \intdm{|\zb| > \frac{|\bx|}{2} }{ \varrho(|\zb|) \frac{|\bu(\zb+\bx)|}{|\zb|} }{\zb} \\
    &\qquad + \intdm{|\zb| > \frac{|\bx|}{2} }{ \varrho(|\zb|) \frac{|\bu(\bx)|}{|\zb|} }{\zb}
    \end{split}
\end{equation*}
for some $t \in [0,1]$. We use the definition of $[\bu]_{\alpha,\beta}$ to estimate:
\begin{equation*}
    \begin{split}
    |\cZ_{\varrho} \bu(\bx)|
    &\leq \intdm{|\zb|\leq \frac{|\bx|}{2} }{ \varrho(|\zb|) \frac{[\bu]_{1,j} }{|\bx+t(\zb)|^j} }{\zb} + \intdm{|\zb| > \frac{|\bx|}{2} }{ \varrho(|\zb|) \frac{|\bu(\zb+\bx)|}{ |\zb|} }{\zb} \\
    &\qquad + \intdm{|\zb| > \frac{|\bx|}{2} }{ \varrho(|\zb|) \frac{[\bu]_{0,k}}{|\bx|^k |\zb|} }{\zb}\,.
    \end{split}
\end{equation*}
Since $\varrho(\bseta) |\bseta|^{-1}$ is nonincreasing and $|\bx + t\zb| \geq \frac{|\bx|}{2}$ for all $\zb \in B(0,{|\bx|\over 2})$,
\begin{equation*}
    \begin{split}
    |\cZ_{\varrho} \bu(\bx)|
    &\leq C(j) \frac{[\bu]_{1,j} }{|\bx|^j} \intdm{|\zb|\leq \frac{|\bx|}{2} }{ \varrho(|\zb|) }{\zb} + \frac{\varrho(\frac{|\bx|}{2})}{ \left( \frac{|\bx|}{2} \right) } \intdm{|\zb| > \frac{|\bx|}{2} }{ |\bu(\zb+\bx)| }{\zb} \\
    &\qquad +  \frac{[\bu]_{0,k}}{|\bx|^k} \intdm{|\zb| > \frac{|\bx|}{2} }{ \frac{\varrho(|\zb|)}{|\zb|}  }{\zb}\,.
    \end{split}
\end{equation*}
Changing coordinates gives \eqref{eq:DecayRate}.
\end{proof}
\begin{remark}
For $m \in \bbN$, denoting the class of $m$-times continuously differentiable and bounded vector fields on $\bbR^d$ by $C^m_b(\bbR^d;\bbR^d)$, the first part of the above proof implies that if $\bu \in C^{m}_b(\bbR^d;\bbR^d)$  then $\cZ_{\varrho}\bu(\bx)$ is well-defined for all $\bx \in \bbR^d$ and $\cZ_{\varrho} \bu \in C^{m-1}_b(\bbR^d;\bbR^d)$ with the estimate 
    \[
    \|D^{\gamma}\cZ_{\varrho} \bu\|_{L^{\infty}(\bbR^d)} \leq C \left[ \|\grad D^{\gamma}{\bu}\|_{L^{\infty}(\bbR^d)}\|\varrho\|_{L^{1}(B_1(0))} + \|D^{\gamma}\bu\|_{L^{\infty}(\bbR^d)}\left\|{\varrho(\cdot)\over |\cdot|}\right\|_{L^1 (\bbR^d\setminus B_1(0))} \right]
    \]
    for any $1\leq |\gamma|\leq m-1.$
\end{remark}

We conclude the section by giving three examples of kernels that satisfy \eqref{assumption:Kernel} and demonstrate the decay estimates in Proposition \ref{decay-estimates-operator}. We will refer to these examples in the sequel. 

\begin{example}\label{ex:frac_kernel}
\textbf{The Fractional Kernel}
For $s \in (0,1)$ define the \textit{fractional kernel}
\begin{equation}\label{eq:Def:FracKernel}
    \varrho_s(|\bseta|) := \frac{c_{d,s}}{|\bseta|^{d+s-1}}\,, \qquad  c_{d,s} := \frac{2^s \Gamma(\frac{d+s+1}{2})}{\pi^{d/2} \Gamma(\frac{1-s}{2})}\,,  \qquad \bseta \in \bbR^d \setminus \{ {\bf 0} \}\,.
\end{equation}
Choosing $j = d+1$ and $k = d$ in \eqref{eq:DecayRate}, and computing the integrals gives the decay rate
\begin{equation*}
    |\cZ_{\varrho_s} \bu(\bx)| \leq \frac{C}{|\bx|^{d+s}}\,.
\end{equation*}
where $C$ depends on $d$, $s$ and $\bu$.
\end{example}

\begin{example}
\textbf{The Truncated Fractional Kernel.} Let $\delta > 0$. Define the \textit{truncated fractional kernel} $\varrho_{s,\delta}$ by
\begin{equation*}
    \varrho_{s,\delta}(|\bseta|) := c_{d,s} \frac{\chi_{B({\bf 0},\delta)} (|\bseta|) }{|\bseta|^{d+s-1}}\,, \qquad \bseta \in \bbR^d \setminus \{ {\bf 0} \}\,.
\end{equation*}
Using the notation $\cZ_{s, \delta}$ for $\cZ_{\varrho_{s,\delta}}$, the estimate \eqref{eq:DecayRate} corresponding to this kernel becomes
\begin{equation}\label{decay-Truncated-frac-ker}
    \begin{split}
    |\cZ_{s,\delta} \bu(\bx)| &\leq \frac{C \delta^{1-s}}{|\bx|^j} \,, \qquad 2 \delta \leq |\bx|
    \end{split}
\end{equation}
for any $j \in \bbN$, where $C$ depends on $d$, $s$ and $\bu$.
\end{example}

\begin{example}
\textbf{The Tempered Fractional Kernel.}
Let $\alpha > 0$, and let $s \in (0,1)$. We define the \textit{tempered fractional kernel}
\begin{equation*}
    \varrho_{s,\text{temp}}(|\bseta|) := \frac{\rme^{-\alpha|\bseta|}}{|\bseta|^{d+s-1}}\,, \qquad \bseta \in \bbR^d\,.
\end{equation*}
We abbreviate the operators $\cZ_{\varrho_{s,\text{temp}}}$ as $\cZ_{s,\text{temp}}$. The exponential decay of $\varrho_{s,\text{temp}}$ gives the resulting nonlocal derivatives rapid decay. To see this, we consider the three terms in \eqref{eq:DecayRate} separately. First, integrating directly we have
\begin{equation}\label{eq:DecayRate:TemperedFrac:Pf1}
     \frac{1}{|\bx|^j} \intdm{|\bh| \leq \frac{|\bx|}{2} }{ \frac{\rme^{-\alpha|\bseta|}}{|\bseta|^{d+s-1}} }{\bh} = \frac{\omega_{d-1}}{|\bx|^j} \int_0^{|\bx|/2} \frac{\rme^{-\alpha r}}{r^s} \, \rmd r \leq \frac{C(d,s,\alpha) \Gamma(1-s) }{|\bx|^j} \text{ for all } |\bx| \geq 1\,.
\end{equation}
Next, by change of coordinates
\begin{equation}\label{eq:DecayRate:TemperedFrac:Pf2}
    \begin{split}
     \frac{1}{|\bx|^k} \intdm{|\bh| > \frac{|\bx|}{2} }{ \frac{\rme^{-\alpha |\bh|}}{|\bh|^{d+s}}  }{\bh} = \frac{\omega_{d-1}}{|\bx|^k} \int_{\frac{|\bx|}{2} }^{\infty} \frac{\rme^{-\alpha r}}{r^{1+s}} \, \rmd r &= \frac{2^s \omega_{d-1}}{|\bx|^{k+s}} \int_{1}^{\infty} \frac{\rme^{-\frac{\alpha |\bx|}{2} r}}{r^{1+s}} \, \rmd r 
     \leq C \frac{\rme^{-\frac{\alpha |\bx|}{2} }}{|\bx|^{k+1+s}}\,,
    \end{split}
\end{equation}
where we have used the upper bound $\int_1^{\infty} t^{-n}\rme^{-zt} \, \rmd t \leq z^{-1} \rme^{-z}$ in the last inequality (see \cite[Equation 15.1.19]{abramowitz1988handbook}); here $C$ depends on $d$, $s$, $\alpha$ and $\bu$. 
Plugging estimates \eqref{eq:DecayRate:TemperedFrac:Pf1} and \eqref{eq:DecayRate:TemperedFrac:Pf2} in to  \eqref{eq:DecayRate} we arrive at
\begin{equation}\label{eq:DecayRate:TemperedFrac}
    |\cZ_{s,\text{temp}} \bu(\bx)| \leq  C \left( \frac{1}{|\bx|^j} +  \frac{\rme^{-\frac{\alpha |\bx|}{2} }}{|\bx|^{k+1+s}}  + \frac{\rme^{-\frac{\alpha |\bx|}{2}} }{|\bx|^{d+s}} \right) \,, \qquad |\bx| \geq 1\,, \qquad j\,, k \in \bbN\,.
\end{equation}
\end{example}

\begin{remark}
From the decay estimates \eqref{decay-Truncated-frac-ker} and \eqref{eq:DecayRate:TemperedFrac} corresponding to the truncated and the tempered fractional kernels, we see that $\cZ_{s, \delta}$ and $\cZ_{s,\text{temp}}$ map the Schwartz class of vector fields $\scS(\bbR^d)$ into itself. 
\end{remark}



\begin{example}
\textbf{The Characteristic Function Kernel.}
Let $\delta > 0$. We define the \textit{characteristic function kernel}
\begin{equation*}
    \varrho_{\chi,\delta}(|\bseta|) := \frac{d}{ \omega_{d-1} \delta^d} \chi_{B({\bf 0},\delta)}(|\bseta|)\,, \quad \bseta \in \bbR^d\,,
\end{equation*}
where $\omega_{d-1}$ denotes the surface measure of the sphere in $\bbR^d$. Using the notation $\cZ_{\chi,\delta}$ for $\cZ_{\varrho_{\chi,\delta}}$, the estimate \eqref{eq:DecayRate} corresponding to this kernel becomes
\begin{equation*}
    \begin{split}
    |\cZ_{\chi,\delta} \bu(\bx)| &\leq \frac{C}{|\bx|^j} \,, \qquad 2 \delta \leq |\bx|
    \end{split}
\end{equation*}
for any $j \in \bbN$, where $C$ depends on $d$ and $\bu$.
\end{example}

\section{H\"older spaces and fractional vector calculus}\label{sec:holder}
The mapping properties of the nonlocal operators $\cG_\varrho$, $\cD_\varrho$, and $\cC_\varrho$ depend on the kernel $\varrho$. In the case of the fractional kernel \eqref{eq:Def:FracKernel}, it is possible to characterize the mapping properties of these operators completely for several function spaces. We refer to the nonlocal gradient, divergence, and curl operators associated with the fractional kernel \eqref{eq:Def:FracKernel} as the fractional gradient, divergence, and curl, respectively, and identify them using the noation 
\begin{equation*}
    \cG_{s} := \cG_{\varrho_s}\,, \qquad \cD_{s} := \cD_{\varrho_s}\,, \qquad  \cC_{s} := \cC_{\varrho_s}\,. 
\end{equation*}
The mapping properties of $\cG_s$ and $\cD_s$ in fractional Sobolev spaces were established by \cite{d2020unified}, and are analogous to the well-known mapping property of the fractional Laplacian $(-\Delta)^s$ in Sobolev spaces \citep{lischke2018fractional, stein2016singular}. In this section, we study the mapping properties in H\"older spaces of these operators. The properties will be used in in Section \ref{sec:identities} to prove identities for fractional vector calculus operators in larger spaces than for general nonlocal operators, and in Section \ref{sec:helmholtz} in proving a Helmholtz decomposition involving fractional operators in H\"older spaces.  

We will define the fractional gradient operators for functions that satisfy appropriate smoothness and integrability conditions. For $\alpha \in (0,2)$, we define the weighted Lebesgue space $L^1_{\alpha}$ as
\begin{equation*}
    L^1_{\alpha}(\bbR^d;\bbR^d) := \left\{ \bu \in L^1_{loc}(\bbR^d;\bbR^d) \, : \, \Vnorm{\bu}_{L^1_{\alpha}(\bbR^d)} := \intdm{\bbR^d}{ \frac{|\bu(\bx)|}{1+|\bx|^{d+\alpha}} }{\bx} < \infty \right\}\,.
\end{equation*}
Note that for any $\alpha \in (0,2)$, $L^p(\bbR^d;\bbR^d) \subset L^1_{\alpha}(\bbR^d;\bbR^d)$ for $p \in [1,\infty]$.

\begin{theorem}\label{thm:MappingPropertiesOfOperators}
    Let $s \in (0,1)$, $N \geq 1$, and $d \geq 2$. Let $\cZ_s \bu(\bx)$ denote any of the following objects:
    \begin{equation*}
        \begin{split}
            \cG_s \bu(\bx)\,, &\quad \text{ for } \bu \in L^1_s(\bbR^d;\bbR^N)\,, \\
            \cD_s \bu(\bx)\,, &\quad \text{ for } \bu \in L^1_s(\bbR^d;\bbR^{N \times d})\,, \\
            \cC_s \bu(\bx)\,, &\quad \text{ for } \bu \in L^1_s(\bbR^d;\bbR^d) \text{ and } d =3\,.
        \end{split}
    \end{equation*}
    Then we have the following:
    \begin{itemize}
        \item[1)] If $\bu \in C^{0,\beta}(\bbR^d)$ for $\beta \in (s,1)$, then $\cZ_s \bu \in C^{0,\beta-s}(\bbR^d)$ with
        \begin{equation}\label{eq:NaturalEstimate1}
            \Vnorm{\cZ_s \bu}_{C^{0,\beta-s}(\bbR^d)}
            \leq
            C \Vnorm{\bu}_{C^{0,\beta}(\bbR^d)}\,.
        \end{equation}
        \item[2)] If $\bu \in C^{1,\beta}(\bbR^d)$ for $\beta \in (0,1)$ and $s < \beta$, then $\cZ_s \bu \in C^{1,\beta-s}(\bbR^d)$ with
        \begin{equation}\label{eq:NaturalEstimate2}
            \Vnorm{\cZ_s \bu}_{C^{1,\beta-s}(\bbR^d)}
            \leq
            C \Vnorm{\bu}_{C^{1,\beta}(\bbR^d)}\,.
        \end{equation}
        \item[3)] If $\bu \in C^{1,\beta}(\bbR^d)$ for $\beta \in (0,1)$ and $s > \beta$, then $\cZ_s \bu \in C^{0,\beta-s+1}(\bbR^d)$ with
        \begin{equation}\label{eq:NaturalEstimate3}
            \Vnorm{\cZ_s \bu}_{C^{0,\beta-s+1}(\bbR^d)}
            \leq
            C \Vnorm{\bu}_{C^{1,\beta}(\bbR^d)}\,.
        \end{equation}
    \end{itemize}
    In all estimates the constant $C$ depends only on $d$, $N$, $s$ and $\beta$.
\end{theorem}

\begin{proof}

To prove 1), we write
\begin{align*}
|\cZ_s \ub(\xb)|
&\le
\int_{\mathbb{R}^d}
\frac{|\ub(\yb) - \ub(\xb)|}{|\yb-\xb|^{d+s}} \, \rmd \yb 
\\
&\leq
\int_{|\yb-\xb| \leq R}
\frac{|\ub(\yb) - \ub(\xb)|}{|\yb-\xb|^{d+s}} \, \rmd \yb
+
\int_{|\yb-\xb| > R}
\frac{|\ub(\yb) - \ub(\xb)|}{|\yb-\xb|^{d+s}} \, \rmd \yb 
\\
&\leq
\left[ \ub \right]_{C^{0,\beta}(\mathbb{R}^d)}
\int_{|\yb-\xb| \leq R}
\frac{1}{|\yb-\xb|^{d+s-\beta}} d\yb
+
2 \| \ub \|_{L^\infty(\mathbb{R}^d)}
\int_{|\yb-\xb| > R}
\frac{1}{|\yb-\xb|^{d+s}} \, \rmd \yb
\\
&\leq
C
R^{\beta - s}
\left[ \ub \right]_{C^{0,\beta}(\mathbb{R}^d)}
+
2 R^{-s} 
\| \ub \|_{L^\infty(\mathbb{R}^d)}
\end{align*}
for any $R>0$. This holds for all $\xb \in \mathbb{R}^d$, so
$\cZ_s \ub \in L^\infty(\mathbb{R}^d)$. 

Next, we use the following notation for the placeholder $\cZ_s$:
\begin{equation*}
\cZ_s \ub(\xb) - \cZ_s \ub(\yb)
= 
\int_{\mathbb{R}^d}
\frac{\left( \ub(\xb+\hb) - \ub(\xb) \right) - \left( \ub(\yb+\hb) - \ub(\yb) \right)}{|\hb|^{d+s}}
\left( \otimes, \times, \cdot \right)
\frac{\hb}{|\hb|} \, \rmd \hb.
\end{equation*}
It is clear that for any $R > 0$
\begin{align*}
|\cZ_s \ub(\xb) - \cZ_s \ub(\yb)|
&\leq 
\int_{\mathbb{R}^d}
\frac{| \left( \ub(\xb+\hb) - \ub(\xb) \right) - \left( \ub(\yb+\hb) - \ub(\yb) \right) |}{|\hb|^{d+s}}
\, \rmd \hb
\\
&=
\int_{|\bh| \leq R} ...
+
\int_{|\bh| > R} ...
\\
&=
I + II. 
\end{align*}
To estimate $I$, we use
\begin{equation*}
|\ub(\zb+\hb) - \ub(\zb)| \leq [\ub]_{C^{0,\beta}(\mathbb{R}^d)} |\hb|^\beta \text{ for } \bz = \bx \text{ or } \by\,.
\end{equation*}
Therefore, 
\begin{align*}
I 
&\le
\left[ \bu \right]_{C^{0,\beta}(\mathbb{R}^d)}
\int_{|\hb|\leq R}\frac{1}{|\hb|^{d+s-\beta}} \, \rmd\hb
=
C
\left[ \ub \right]_{C^{0,\beta}(\mathbb{R}^d)}
R^{\beta - s}. 
\end{align*}
For II, we use the estimate that, for $\bh\in \mathbb{R}^d$,
\begin{equation*}
    |\bu(\bx+\bh)-\bu(\by+\bh)| 
        + |\bu(\bx)-\bu(\by)| 
    \leq 2 [\ub]_{C^{0,\beta} (\mathbb{R}^d)} 
    |\bx-\by|^{\beta}
\end{equation*}
to obtain
\begin{align*}
II &\leq 2 [\ub]_{C^{0,\beta} (\mathbb{R}^d)}
\int_{|\hb| \leq R }
\frac{|\xb-\yb|^{\beta}}{|\hb|^{d+s}}  \, \rmd\hb
\\
&=
C
\left[ \ub \right]_{C^{0,\beta}(\mathbb{R}^d)}
\frac{|\xb-\yb|^\beta}{R^s}\,.
\end{align*}
Choosing $R = |\xb-\yb|$ gives
\begin{align*}
|\cZ_s \ub(\xb) - \cZ_s \ub(\yb)| &\leq I + II
\leq C 
\left[ \ub \right]_{C^{0,\beta}(\mathbb{R}^d)}
|\xb-\yb|^{\beta - s}. 
\end{align*}
Therefore, $\cZ_s \ub \in C^{0,\beta-s} (\mathbb{R}^d)$. 


The proof of 2) proceeds in the same as for the proof of 1), but with $\nabla \ub$ in place of $\ub$. Here, one only needs to verify that the operator $\cZ_s$ commutes with derivatives. The process to verify this follows identically to the process in the proof of \Cref{decay-estimates-operator}, with the estimate 
\eqref{eq:L1Est2} replaced with
\begin{equation*}
    \frac{|D_i \bu(\bx+\bh)-D_i \bu(\bx)|}{|\bz|^{d+s}} 
    \leq \chi_{ \{|\bh|\leq 1\} } 
        \frac{ [ \grad \bu ]_{C^{0,\beta}(\bbR^d)} }{ |\bh|^{d+s-\beta} }
        + 2 \chi_{ \{|\bh|\geq 1\} } 
        \frac{ \Vnorm{\grad \bu}_{L^{\infty}(\bbR^d)} }{ |\bh|^{d+s} }\,.
\end{equation*}

To prove 3), we assume $\ub \in C^{1,\beta}(\mathbb{R}^d)$ for $s > \beta$. To show that
$\cZ_s \ub \in C^{0,\beta - s + 1}(\mathbb{R}^d)$, we write
\begin{align*}
|\cZ_s \ub(\xb)|
&\le 
\int_{\mathbb{R}^d}
\frac{|\ub(\yb) - \ub(\xb)|}{|\yb - \xb|^{d+s}}
\, \rmd \yb
\\
&\le
\|\nabla \ub\|_{L^{\infty}(\bbR^d)}
\int_{|\yb - \xb| < R}
\frac{1}{|\yb - \xb|^{d+s-1}}
\, \rmd \yb
+
2 \| \ub \|_{L^\infty(\mathbb{R}^d)}
\int_{|\yb - \xb| \ge R} \frac{1}{|\yb - \xb|^{d+s}}
\, \rmd \yb
\\
&\le
\|\ub\|_{C^1(\mathbb{R}^d)}
\left(
R^{1-s} + R^{-s}
\right). 
\end{align*}
Thus, $\cZ \ub \in L^{\infty}(\mathbb{R}^d)$. 
Next, we have
\begin{align*}
|\cZ_s \ub(\xb) - \cZ_s \ub(\yb)|
&\le
\int_{\mathbb{R}^d}
\frac{|(\ub(\xb + \hb) - \ub(\xb)) - (\ub(\yb + \hb) - \ub(\yb))|}
{|\hb|^{d+s}} \, \rmd \hb
\\
&=
\int_{|\hb| \le R} ... + \int_{|\hb| > R} ...
\\
&=
I + II. 
\end{align*}
For $I$, we know that for $t, t' \in (0,1)$, and any $\bh\in \mathbb{R}^d$ we write first, using mean value theorem, 
\[
[\ub(\xb+\hb) - \ub(\xb)]
-
[\ub(\yb+\hb) - \ub(\yb)]
=
\nabla \ub(\xb + t \hb) \cdot \hb
-
\nabla \ub(\yb + t' \hb) \cdot \hb
\]
We add and subtract terms $\nabla \ub(\xb) \cdot \hb$ and $\nabla \ub(\yb) \cdot \hb$ to obtain the estimate 
\begin{align*}
&|(\ub(\xb+\hb) - \ub(\xb))
-
(\ub(\yb+\hb) - \ub(\yb))|\\
&\le |\nabla \ub(\xb) - \nabla \ub(\yb)| \, |\hb|
+ |\nabla \ub(\xb + t \hb) - \nabla \ub(\xb)| \, |\hb|
+ |\nabla \ub(\xb + t' \hb) - \nabla \ub(\yb)| \, |\hb|
\\
&\le
|\nabla \ub(\xb) - \nabla \ub(\yb)| \, |\hb|
+ 2 \left[ \nabla \ub \right]_{C^{0,\beta}(\mathbb{R}^d)} |\hb|^{1+\beta}
\\
&\le
2 [\nabla \ub]_{C^{0,\beta}(\mathbb{R}^d)}
\left(
|\xb - \yb|^{\beta}|\hb| + |\hb|^{1+\beta}
\right).
\end{align*}
Thus, 
 \begin{align*}
I &\le 2[\nabla \ub]_{C^{0,\beta}(\mathbb{R^d})}
\left[
\int_{|\hb| \le R}
|\xb - \yb|^\beta \frac{1}{|\hb|^{d+s-1}}
+
\frac{1}{|\hb|^{d+s-\beta-1}}
 \, \rmd \hb
\right]
\\
&=
C [\nabla \ub]_{C^{0,\beta}(\mathbb{R^d})}
\left(
|\xb - \yb|^{\beta} R^{1-s}
+
R^{1+\beta-s}
\right).
\end{align*}
To estimate $II$, we have, for some $t, t' \in (0,1)$, for any $\bh$
\begin{align*}
[\ub(\xb&+\hb) - \ub(\yb+\hb)] - [\ub(\xb) - \ub(\yb)]
\\
&=\nabla \ub\big(t \xb + t \hb + (1-t) \yb + (1-t) \hb\big) (\xb - \yb)
- \nabla \ub\big(t' \xb + (1-t') \yb \big) (\xb - \yb)
\\
&=
\left[ \nabla \ub \big( \yb + t (\xb - \yb) + \hb \big)
-
\nabla \ub \big( \yb + t' (\xb - \yb) \big) \right](\xb - \yb)
\end{align*}
We now add and subtract appropriate terms to be able to write
\begin{align*}
[\ub(\xb&+\hb) - \ub(\yb+\hb)] - [\ub(\xb) - \ub(\yb)]
\\
&=
\left[ \nabla \ub \big( \yb + t (\xb - \yb) + \hb \big)
-
\nabla \ub \big( \yb + t' (\xb - \yb) \big) \right](\xb - \yb)\\
&=\left[ \nabla \ub \big( \yb + t (\xb - \yb) + \hb \big)
-\nabla \ub \big( \xb + t (\xb - \yb) +\hb \big) \right] (\xb - \yb)\\
&+\left[\nabla \ub \big( \xb + t (\xb - \yb) +\hb \big)-\nabla \ub(\xb + \bh)\right](\xb-\yb)\\
&+\left[\nabla \ub(\xb+\hb) - \nabla \ub(\bx)\right](\xb-\yb)\\
&+\left[\nabla \ub(\xb) -\nabla \ub (\yb) \right](\xb-\yb)\\
&+\left[\nabla \ub(\yb) - \nabla \ub(\yb + t'(\xb-\yb))\right](\xb-\yb).
\end{align*}
Therefore, estimating each term as before using the H\"older continuity of $\ub$ we have 
\[
|[\ub(\xb+\hb - \ub(\yb+\hb)] - [\ub(\xb) - \ub(\yb)]| \leq C \left[ \nabla \ub \right]_{C^{0,\beta}(\mathbb{R}^d)} \left( |\yb - \xb|^{\beta+1} + |\xb - \yb|
|\hb|^{\beta} \right).
\]
Thus, 
\begin{equation*}
II 
\le
C [\nabla \ub]_{C^{0,\beta}(\mathbb{R}^d)}
\left(
\int_{|\xb - \yb| > R} 
\frac{|\xb-\yb|^{1+\beta}}{|\hb|^{d+s}}
+
\frac{|\xb - \yb|}{|\hb|^{d+s-\beta}} 
\, \rmd \hb
\right).
\end{equation*}
Since $s > \beta$, the above integral exists and so
\begin{equation*}
II \le 
C [\nabla \ub]_{C^{0,\beta}(\mathbb{R}^d)}
\left(
\frac{|\xb - \yb|^{1+\beta}}{R^s}
+
\frac{|\xb - \yb|}{R^{s-\beta}}
\right).
\end{equation*}
Putting I and II together, we obtain that for any $R>0$, and $\xb,\yb\in \mathbb{R}^d$ 
\begin{align*}
|\cZ_s \ub (\xb) - \cZ_s \ub(\yb)|
&\le 
I + II 
\\
&\le
C [\nabla \ub]_{C^{0,\beta}(\mathbb{R^d})} 
\left(
|\xb - \yb|^{\beta} R^{1-s}
+
R^{1+\beta-s}
+
\frac{|\xb - \yb|^{1+\beta}}{R^s}
+
\frac{|\xb - \yb|}{R^{s-\beta}}
\right).
\end{align*}
Choosing $R = |\xb - \yb|$ gives us
\begin{equation*}
| \cZ_s \ub(\xb) - \cZ_s \ub(\yb) |
\le
C [\nabla \ub]_{C^{0,\beta}(\mathbb{R}^d)}
|\xb - \yb|^{\beta - s + 1},
\end{equation*}
completing the proof. 

\end{proof}

\section{Vector calculus identities for nonlocal operators}\label{sec:identities}

{This section is devoted to the proof of several operator identities whose local, classical counterpart is well-established, but that have not been fully investigated for the nonlocal operators considered in this work. To handle the potential singularity along the diagonal $\bx=\by$, we start by proving similar identities for ``truncated'' operators first, and then recover the desired identities in the vanishing truncation limit. The latter is justified by an important result proved at the beginning of this section in Theorem \ref{thm:OperatorsWellDefdForSmoothFxns}. This allows us to establish the validity of the operator identities for bounded $C^2$ functions.}

{We define the truncated operators below; note that in the nonlocal literature (see, e.g. \cite{Delia2013}) ``truncated'' operators usually correspond to ``compactly supported'' kernels; however, in our usage below, the truncation is performed in a neighborhood of $\xb$, i.e. we remove from the domain of integration an infinitesimal ball centered at $\xb$. Let $\veps > 0$. The truncated gradient, divergence and curl operators are defined as}
\begin{align*}
\cG_{\varrho,\veps}\ub(\xb) &= \int_{\bbR^d \setminus B(\bx,\veps)} \varrho (\yb-\xb)  \frac{(\ub(\yb)-\ub(\xb))}{|\by-\bx|} \otimes \frac{\by-\bx}{|\by-\bx|}  \, \rmd\yb\,, \qquad \bu : \bbR^d \to \bbR^N\,,\\
\cD_{\varrho,\veps}\ub(\xb) &= \int_{\bbR^d  \setminus B(\bx,\veps)} \varrho (\yb-\xb)  \frac{(\ub(\yb)-\ub(\xb))}{|\by-\bx|} \frac{\by-\bx}{|\by-\bx|} \, \rmd\yb\,, \qquad \bu : \bbR^d \to \bbR^{N \times d}\,, \\
\cC_{\varrho,\veps}\ub(\xb) &= \int_{\bbR^d  \setminus B(\bx,\veps)} \varrho (\yb-\xb) \frac{\by-\bx}{|\by-\bx|} \times \frac{(\ub(\yb)-\ub(\xb))}{|\by-\bx|} \, \rmd\yb\,, \quad \bu : \bbR^d \to \bbR^d \text{ and } d = 3\,.
\end{align*}

We use the notation $\cZ_{\varrho,\veps} \bu(\bx)$ in exactly the same way as in Proposition \ref{decay-estimates-operator}.
Note that for $\bu \in L^{\infty}(\bbR^d)$, we have that
\begin{equation*}
    | \cZ_{\varrho,\veps} \bu(\bx)| \leq 2 \Vnorm{\bu}_{L^{\infty}(\bbR^d)} \intdm{|\bh| \geq \veps }{\frac{\varrho(|\bh|)}{|\bh|} }{\bh}.
\end{equation*}
Thus, for any fixed $\veps > 0$ all three operators are well-defined. {The next theorem shows that the composition of the limits of two operators, when compatible, equals the limit of the truncated composition.}

\begin{theorem}\label{thm:OperatorsWellDefdForSmoothFxns}
    Let $d \geq 2$ and $N \geq 1$.
    Let $\cY_{\varrho} \circ \cZ_{\varrho} \bu(\bx)$ denote any of the following compositions of operators:
    \begin{equation*}
        \begin{split}
            \cG_{\varrho} \circ \cG_{\varrho} u(\bx)\,, &\qquad
            u : \bbR^d \to \bbR\,,\\
            \cD_{\varrho} \circ \cG_{\varrho} \bu(\bx)\,, &\qquad
            \bu : \bbR^d \to \bbR^N\,,\\
            \cC_{\varrho} \circ \cG_{\varrho} u(\bx)\,, &\qquad
            u : \bbR^d \to \bbR \text{ and } d = 3\,,\\
            \cG_{\varrho} \circ \cD_{\varrho} \bu(\bx)\,, &\qquad
            \bu : \bbR^d \to \bbR^{N\times d}\,,\\
            \cD_{\varrho} \circ \cD_{\varrho} \bu(\bx)\,, &\qquad
            \bu : \bbR^d \to \bbR^{d \times d}\,,\\
            \cC_{\varrho} \circ \cD_{\varrho} \bu(\bx)\,, &\qquad
            \bu : \bbR^d \to \bbR^{d \times d} \text{ and } d = 3\,,\\
            \cG_{\varrho} \circ \cC_{\varrho} \bu(\bx)\,, &\qquad
            \bu : \bbR^d \to \bbR^d \text{ and } d = 3\,,\\
            \cD_{\varrho} \circ \cC_{\varrho} \bu(\bx)\,, &\qquad
            \bu : \bbR^d \to \bbR^d \text{ and } d = 3\,,\\
            \cC_{\varrho} \circ \cC_{\varrho} \bu(\bx)\,, &\qquad
            \bu : \bbR^d \to \bbR^d \text{ and } d = 3\,.\\
        \end{split}
    \end{equation*}
    If either
    \begin{enumerate}
    \item[1)]
    $\bu \in C^2_b(\bbR^d)$, or
    \item[2)]
    $\varrho = \varrho_s$ and $\bu \in L^1_{2s}(\bbR^d) \cap \scC^{2s+\sigma}(\bbR^d)$ for $\sigma > 0$ sufficiently small,
    \end{enumerate}
    then $\cY_{\varrho} \circ \cZ_{\varrho} \bu(\bx)$ is a bounded function. Furthermore, we have
    \begin{equation*}
        \cY_{\varrho} \circ \cZ_{\varrho} \bu(\bx) = \lim_{\veps, \veps' \to 0}  \cY_{\varrho,\veps} \circ \cZ_{\varrho,\veps'} \bu(\bx)
    \end{equation*}
    where $ \cY_{\varrho,\veps}$, $ \cZ_{\varrho,\veps'}$ denote the relevant truncated form of the operator.
\end{theorem}

\begin{proof}
For any $\veps$, $\veps' > 0$ we have
\begin{equation*}
    \begin{split}
    |\cY_{\varrho,\veps} \circ \cZ_{\varrho,\veps'} \bu(\bx)| &\leq \intdm{\bbR^d \setminus B({\bf 0},\veps)}
        {
        \varrho( \bf h )  \frac{|\cZ_{\varrho,\veps'} \bu(\bx+\bh) - \cZ_{\varrho,\veps'} \bu(\bx)| 
        }{|\bh|} }{\bh}\,.
    \end{split}
\end{equation*}
We will use the Lebesgue Dominated Convergence Theorem. We will derive the relevant estimates for the function
\begin{equation*}
    \Upsilon_{\veps,\veps'}(\bx,\bh) :=
    \chi_{\bbR^d \setminus B({\bf 0},\veps)} \varrho(\bh) \frac{|\cZ_{\varrho,\veps'} \bu(\bx+\bh) - \cZ_{\varrho,\veps'} \bu(\bx)| }{|\bh|}\,.
\end{equation*}
Specifically, we will show that there exists a function $\Upsilon(\bh)$ such that $\Upsilon \in L^1(\bbR^d)$ and
\begin{equation*}
    |\Upsilon_{\veps,\veps'}(\bx,\bh)| \leq |\Upsilon(\bh)| \quad \text{ for all } \bx\,, \bh \in \bbR^d\,, \qquad \text{ for all } \veps,\veps' > 0\,.
\end{equation*}
First we prove the theorem for case 1). We have
\begin{equation*}
    \begin{split}
        \Upsilon_{\veps,\veps'}(\bx,\bh) \leq \varrho(|\bh|) \min \left\{ \Vnorm{\grad \cZ_{\varrho,\veps'} \bu}_{L^{\infty}(\bbR^d)} \,, \frac{ \Vnorm{\cZ_{\varrho,\veps'} \bu}_{L^{\infty}(\bbR^d)} }{|\bh|} \right\}\,.
    \end{split}
\end{equation*}
Therefore, it suffices to show that there exist constants $b_1$ and $b_2$ independent of $\veps'$ such that
\begin{equation}\label{eq:CurlOfCurl:SmoothFxns:Pf1}
    \Vnorm{\cZ_{\varrho,\veps'} \bu}_{L^{\infty}(\bbR^d)} \leq b_1\,, \qquad \Vnorm{\grad \cZ_{\varrho,\veps'} \bu}_{L^{\infty}(\bbR^d)} \leq b_2\,, \qquad \text{ for all } \veps' >0
\end{equation}
and the proof will be complete by setting $\Upsilon(\bh) = \varrho(|\bh|) \min \left\{ b_2 \,, \frac{b_1 }{|\bh|} \right\} $.
To prove \eqref{eq:CurlOfCurl:SmoothFxns:Pf1} we proceed analogously to  \eqref{eq:L1Est1}:
\begin{equation*}
    \begin{split}
    |\cZ_{\varrho,\veps'} \bu(\bx)| &\leq \intdm{\veps' \leq |\bz|\leq 1 }{ \varrho(|\bz|) \frac{|\bu(\bx+\bz)-\bu(\bx)|}{|\bz|} }{\bz} + \intdm{|\bz| > 1 }{ \varrho(|\bz|) \frac{|\bu(\bx+\bz)-\bu(\bx)|}{|\bz|} }{\bz} \\
    &\leq \intdm{|\bz|\leq 1 }{ \varrho(|\bz|) \frac{|\grad \bu(\bx)(\bz)| + o(|\bz|)}{|\bz|} }{\bz} + \intdm{|\bz| > 1 }{ \varrho(|\bz|) \frac{|\bu(\bx+\bz)|+|\bu(\bx)|}{|\bz|} }{\bz} \\
    &\leq |\grad \bu(\bx)| \intdm{|\by-\bx|\leq 1 }{ \varrho(|\bz|) }{\bz} + \Vnorm{\bu}_{L^{\infty}(\bbR^d)} \intdm{|\bz| > 1 }{  \frac{\varrho(|\bz|)}{|\bz|} }{\bz} \\
    &:= b_1 < \infty\,.
    \end{split}
\end{equation*}
The estimate for $\grad \cZ_{\varrho,\veps'} \bu$ follows the same lines, since the operators $\cZ_{\varrho,\veps'}$ commute with derivatives. Therefore, the theorem is proved for case 1).

{
For case 2) and for $s < 1/2$, we need to show that
\begin{equation*}
    \Upsilon_{\veps,\veps'}(\bx,\bh) = \chi_{\bbR^d \setminus B({\bf 0},\veps)} \frac{|\cZ_{s,\veps'} \bu(\bx+\bh) - \cZ_{s,\veps'} \bu(\bx)| }{|\bh|^{d+s}}\,.
\end{equation*}
is bounded by an $L^1$ function $\Upsilon(\bh)$. If we can show the existence of constants $b_1$ and $b_2$ independent of $\veps'$ such that
\begin{equation}\label{eq:Composition:FractionalOp:Proof1}
    \Vnorm{\cZ_{s,\veps'} \bu}_{L^{\infty}(\bbR^d)} \leq b_1\,,
    \qquad [\cZ_{s,\veps'} \bu]_{C^{0,s+\sigma}(\bbR^d)} \leq b_2\,,
\end{equation}
then we have the upper bound by an $L^1$ function
\begin{equation*}
    \Upsilon_{\veps,\veps'}(\bx,\bh) 
    \leq \chi_{ \{ |\bh|\leq 1 \} } 
        \frac{[\cZ_{s,\veps'} \bu]_{C^{0,s+\sigma}(\bbR^d)}}{|\bh|^{d-\sigma}}
    + 2 \chi_{ \{ |\bh| > 1 \} } 
        \frac{\Vnorm{\cZ_{s,\veps'} \bu}_{L^{\infty}(\bbR^d)}}{|\bh|^{d+s}}\,,
\end{equation*}
and the proof in the case 2) with $s < 1/2$ will be complete.
The existence of $b_1$ and $b_2$ in \eqref{eq:Composition:FractionalOp:Proof1} can be shown by following the proof of the estimate \eqref{eq:NaturalEstimate1} line by line, with $\bu$ replaced by $\cZ_{s,\veps'} \bu$ and $\beta = 2s+\sigma$.

The case 2) and $s \geq 1/2$ is proved the same way, instead following the proof of the estimate \eqref{eq:NaturalEstimate3} line by line.


}
\end{proof}

\subsection{The Curl of the Gradient is Zero}

The following proposition is a nonlocal analogue of the vector calculus identity $\curl \grad u = {\bf 0}$.
\begin{proposition}\label{prop:CurlOfGrad:SmoothFxns}
The identity
\begin{equation}\label{eq:CurlOfGrad:SmoothFxns}
\cC_{\varrho} \circ \cG_{\varrho} u(\bx) = {\bf 0}
\end{equation}
holds for all $\bx \in \bbR^d$ if either
\begin{enumerate}
    \item[1)] $u \in C^2_b(\bbR^d)$ with $d=3$, or
    \item[2)] $\varrho = \varrho_s$ and $u \in L^1_{2s}(\bbR^d) \cap \scC^{2s+\sigma}(\bbR^d)$ for $\sigma > 0$ sufficiently small.
\end{enumerate}
\end{proposition}

This can be shown immediately by applying \Cref{thm:OperatorsWellDefdForSmoothFxns} to the following theorem for the corresponding truncated operators.

\begin{theorem}\label{thm:CurlOfGrad:Truncated}
For any $u \in L^{\infty}(\bbR^d)$ and for any $\veps$, $\veps' > 0$
\begin{equation}\label{eq:CurlOfGrad:Truncated}
    \cC_{\varrho,\veps} \circ \cG_{\varrho,\veps'} u (\xb) = - \cC_{\varrho,\veps'} \circ \cG_{\varrho,\veps} u (\xb)
\end{equation}
for all $\bx \in \bbR^d$.
\end{theorem}

\begin{proof}
Unpacking the operator $\cG_{\varrho,\veps} \circ \cG_{\varrho,\veps'} u$ and changing coordinates,
\begin{equation*}
    \begin{split}
        \cC_{\varrho,\veps} &\circ \cG_{\varrho,\veps'} u \\
            &= \intdm{\bbR^d \setminus B({\bf 0},\veps) }{\varrho(|\bh|) \frac{\bh}{|\bh|} \times \frac{\cG_{\varrho,\veps'} u(\bx+\bh)-\cG_{\varrho,\veps'} u(\xb)}{|\bh|}  }{\bh} \\
            &= \int_{\bbR^d \setminus B({\bf 0},\veps) }  \frac{\varrho(|\bh|)}{|\bh|} \frac{\bh}{|\bh|} \times \Bigg(
         \int_{\bbR^d \setminus B({\bf 0},\veps')} \frac{\varrho(|\bw|)}{|\bw|} \big(u(\bx+\bh+\bw) - u(\bx+\bh) \big) \frac{\bw}{|\bw|} \, \rmd \bw \\
                &\qquad -\int_{\bbR^d \setminus B({\bf 0},\veps')} \frac{\varrho(|\bw|)}{|\bw|} \big(u(\bx+\bw) - u(\bx) \big) \frac{\bw}{|\bw|} \, \rmd \bw \Bigg) \, \rmd \bh \\
            &= \int_{\bbR^d \setminus B({\bf 0},\veps) }  \frac{\varrho(|\bh|)}{|\bh|} \frac{\bh}{|\bh|} \times \\ 
                & \qquad \Bigg( \int_{\bbR^d \setminus B({\bf 0},\veps')} \frac{\varrho(|\bw|)}{|\bw|} \big(u(\bx+\bh+\bw) - u(\bx+\bh) -u(\bx+\bw) + u(\bx) \big) \frac{\bw}{|\bw|} \, \rmd \bw \Bigg) \, \rmd \bh\,.
    \end{split}
\end{equation*}
We are justified in using linearity of the integral in the last equality, since $\bw \mapsto \varrho(\bw) \frac{|\ub(\bx+\bw)-\ub(\bx)|}{|\bw|}$ is in $ L^1(\bbR^d \setminus B({\bf 0},\veps))$ for any $\bu \in L^{\infty}(\bbR^d)$, for any $\veps > 0$ and for any $\bx \in \bbR^d$. Thus we obtain, 
{\small \begin{equation*}
    \cC_{\varrho,\veps} \circ \cG_{\varrho,\veps'} u = \int_{\bbR^d \setminus B({\bf 0},\veps) } \int_{\bbR^d \setminus B({\bf 0},\veps')} \frac{\varrho(|\bh|)}{|\bh|}  \frac{\varrho(|\bw|)}{|\bw|} \big(u(\bx+\bh+\bw) - u(\bx+\bh) -u(\bx+\bw) + u(\bx) \big) \frac{\bh}{|\bh|} \times \frac{\bw}{|\bw|} \, \rmd \bw \, \rmd \bh\,.
\end{equation*} }
The last expression in the double integral is majorized by
\begin{equation}\label{eq:CurlOfGrad:Majorizer}
    C \chi_{ \{|\bh| \geq \veps\} } \chi_{ \{|\bw| \geq \veps'\} } \Vnorm{\bu}_{L^{\infty}(\bbR^d)} \frac{\varrho (\bh)}{|\bh|} \frac{\varrho (\bw)}{|\bw|} \in L^1( \bbR^d \times \bbR^d)\,.
\end{equation}
Therefore, we can use Fubini's theorem and interchange the order of integration:
{\small \begin{equation*}
    \cC_{\varrho,\veps} \circ \cG_{\varrho,\veps'} u = \int_{\bbR^d \setminus B({\bf 0},\veps') } \int_{\bbR^d \setminus B({\bf 0},\veps)} \frac{\varrho(|\bh|)}{|\bh|}  \frac{\varrho(|\bw|)}{|\bw|} \big(u(\bx+\bh+\bw) - u(\bx+\bh) -u(\bx+\bw) + u(\bx) \big) \frac{\bh}{|\bh|} \times \frac{\bw}{|\bw|} \, \rmd \bh \, \rmd \bw\,.
\end{equation*} }
Now, we use the identity $\ba \times \bfb = - (\ba \times \bfb)$, and ``re-pack" the integrals to obtain the result:
{\small \begin{equation*}
    \begin{split}
   & \cC_{\varrho,\veps} \circ \cG_{\varrho,\veps'} u \\
        &= - \int_{\bbR^d \setminus B({\bf 0},\veps') } \int_{\bbR^d \setminus B({\bf 0},\veps)} \frac{\varrho(|\bh|)}{|\bh|}  \frac{\varrho(|\bw|)}{|\bw|} \big(u(\bx+\bh+\bw) - u(\bx+\bh) -u(\bx+\bw) + u(\bx) \big) \frac{\bw}{|\bw|} \times \frac{\bh}{|\bh|} \, \rmd \bh \, \rmd \bw \\
        &= - \int_{\bbR^d \setminus B({\bf 0},\veps') } \frac{\varrho(|\bw|)}{|\bw|} \frac{\bw}{|\bw|} \times \Bigg( 
          \int_{\bbR^d \setminus B({\bf 0},\veps)}  \frac{\varrho(|\bh|)}{|\bh|} \big(u(\bx+\bh+\bw)-u(\bx+\bw) \big) \frac{\bh}{|\bh|} \, \rmd \bh \\
            &\qquad \qquad - \int_{\bbR^d \setminus B({\bf 0},\veps)}  \frac{\varrho(|\bh|)}{|\bh|} \big(u(\bx+\bh)-u(\bx) \big) \frac{\bh}{|\bh|} \, \rmd \bh \Bigg) \, \rmd \bw\\
            &= - \cC_{\varrho,\veps'} \circ \cG_{\varrho,\veps} u(\bx)\,.
    \end{split}
\end{equation*}}
\end{proof}

\begin{proof}[Proof of Proposition \ref{prop:CurlOfGrad:SmoothFxns}]
Use Theorem \ref{thm:OperatorsWellDefdForSmoothFxns} to take the limit as $\veps$, $\veps' \to 0$ on both sides of \eqref{eq:CurlOfGrad:Truncated}:
\begin{equation*}
    \cC_{\varrho} \circ \cG_{\varrho} u = - \cC_{\varrho} \circ \cG_{\varrho} u\,.
\end{equation*}
\end{proof}

\subsection{The Divergence of the Curl is Zero}
We proceed as in the previous section to prove a nonlocal vector calculus analogue of the identity $\div \curl \bu = 0$.

\begin{proposition}\label{prop:DivOfCurl:SmoothFxns}
The identity
\begin{equation}\label{eq:DivOfCurl:SmoothFxns}
\cD_{\varrho} \circ \cC_{\varrho} \bu(\bx) = 0
\end{equation}
holds for all $\bx \in \bbR^d$ if either
\begin{enumerate}
    \item[1)] $\bu \in C^2_b(\bbR^d;\bbR^d)$ with $d=3$, or
    \item[2)] $\varrho=\varrho_s$ and $\bu \in L^1_{2s}(\bbR^d;\bbR^d) \cap \scC^{2s+\sigma}(\bbR^d;\bbR^d)$ with $d=3$ and for $\sigma > 0$ sufficiently small.
\end{enumerate}
\end{proposition}

This can be shown immediately by applying \Cref{thm:OperatorsWellDefdForSmoothFxns} to the following theorem for the corresponding truncated operators.

\begin{theorem}\label{thm:DivOfCurl:Truncated}
For any $\bu \in L^{\infty}(\bbR^d;\bbR^d)$ with $d=3$ and for any $\veps$, $\veps' > 0$
\begin{equation}\label{eq:DivOfCurl:Truncated}
    \cD_{\varrho,\veps} \circ \cC_{\varrho,\veps'} \bu (\xb) = - \cD_{\varrho,\veps'} \circ \cC_{\varrho,\veps} \bu (\xb)
\end{equation}
for all $\bx \in \bbR^d$.
\end{theorem}

\begin{proof}
Unpacking the operator $\cD_{\varrho,\veps} \circ \cC_{\varrho,\veps'} \bu$ and changing coordinates,
{\small \begin{equation*}
    \begin{split}
        &\cD_{\varrho,\veps} \circ \cC_{\varrho,\veps'} \bu \\
            &= \intdm{\bbR^d \setminus B({\bf 0},\veps) }{\varrho(|\bh|) \frac{\cC_{\varrho,\veps'} \bu(\bx+\bh)-\cC_{\varrho,\veps'} \bu(\xb)}{|\bh|} \cdot \frac{\bh}{|\bh|}  }{\bh} \\
            &= \int_{\bbR^d \setminus B({\bf 0},\veps) }  \frac{\varrho(|\bh|)}{|\bh|} \Bigg(
              \int_{\bbR^d \setminus B({\bf 0},\veps')} \frac{\varrho(|\bw|)}{|\bw|} \frac{\bw}{|\bw|} \times \Big( \bu(\bx+\bh+\bw) - \bu(\bx+\bh) \Big) \, \rmd \bw \\
                &\qquad -\int_{\bbR^d \setminus B({\bf 0},\veps')} \frac{\varrho(|\bw|)}{|\bw|} \frac{\bw}{|\bw|} \times \Big(\bu(\bx+\bw) - \bu(\bx) \Big)  \, \rmd \bw \Bigg) \cdot \frac{\bh}{|\bh|} \, \rmd \bh \\
            &= \int_{\bbR^d \setminus B({\bf 0},\veps) }  \frac{\varrho(|\bh|)}{|\bh|} \Bigg(\int_{\bbR^d \setminus B({\bf 0},\veps')} \frac{\varrho(|\bw|)}{|\bw|} \frac{\bw}{|\bw|} \times \big(\bu(\bx+\bh+\bw) - \bu(\bx+\bh) -\bu(\bx+\bw) + \bu(\bx) \big)  \, \rmd \bw \Bigg) \cdot \frac{\bh}{|\bh|} \, \rmd \bh\,.
    \end{split}
\end{equation*}}
We are justified in using linearity of the integral in the last equality, since $\bw \mapsto \varrho(\bw) \frac{|\ub(\bx+\bw)-\ub(\bx)|}{|\bw|}$ is in $L^1(\bbR^d \setminus B({\bf 0},\veps))$ for any $\bu \in L^{\infty}(\bbR^d)$, for any $\veps > 0$ and for any $\bx \in \bbR^d$. Thus we obtain 
{\small \begin{equation*}
    \cD_{\varrho,\veps} \circ \cC_{\varrho,\veps'} \bu = \int_{\bbR^d \setminus B({\bf 0},\veps) } \int_{\bbR^d \setminus B({\bf 0},\veps')} \frac{\varrho(|\bh|)}{|\bh|}  \frac{\varrho(|\bw|)}{|\bw|} \Bigg( \frac{\bw}{|\bw|} \times \big(\bu(\bx+\bh+\bw) - \bu(\bx+\bh) - \bu(\bx+\bw) + \bu(\bx) \big) \Bigg) \cdot \frac{\bh}{|\bh|}  \, \rmd \bw \, \rmd \bh\,.
\end{equation*} }
The last expression in the double integral is majorized by
\begin{equation}\label{eq:DivOfCurl:Majorizer}
    C \chi_{ \{|\bh| \geq \veps\} } \chi_{ \{|\bw| \geq \veps'\} } \Vnorm{\bu}_{L^{\infty}(\bbR^d)} \frac{\varrho (\bh)}{|\bh|} \frac{\varrho (\bw)}{|\bw|} \in L^1( \bbR^d \times \bbR^d)\,. 
\end{equation}
Therefore, we can use Fubini's theorem and interchange the order of integration:
{\small \begin{equation*}
    \cD_{\varrho,\veps} \circ \cC_{\varrho,\veps'} \bu = \int_{\bbR^d \setminus B({\bf 0},\veps') } \int_{\bbR^d \setminus B({\bf 0},\veps)} \frac{\varrho(|\bh|)}{|\bh|}  \frac{\varrho(|\bw|)}{|\bw|} \Bigg( \frac{\bw}{|\bw|} \times \big(\bu(\bx+\bh+\bw) - \bu(\bx+\bh) - \bu(\bx+\bw) + \bu(\bx) \big) \Bigg) \cdot \frac{\bh}{|\bh|}  \, \rmd \bh \, \rmd \bw\,.
\end{equation*}}
Now, we use the identity $(\ba \times \bfb) \cdot \bc = (\bfb \times \bc) \cdot \ba = - (\bc \times \bfb) \cdot \ba$, and ``re-pack" the integrals to obtain the result:
{\small \begin{equation*}
    \begin{split}
    &\cD_{\varrho,\veps} \circ \cC_{\varrho,\veps'} \bu  \\
        &= - \int_{\bbR^d \setminus B({\bf 0},\veps') } \int_{\bbR^d \setminus B({\bf 0},\veps)} \frac{\varrho(|\bh|)}{|\bh|}  \frac{\varrho(|\bw|)}{|\bw|} \Bigg( \frac{\bh}{|\bh|} \times \big(\bu(\bx+\bh+\bw) - \bu(\bx+\bh) - \bu(\bx+\bw) + \bu(\bx) \big) \Bigg) \cdot \frac{\bw}{|\bw|}  \, \rmd \bh \, \rmd \bw \\
        &= - \int_{\bbR^d \setminus B({\bf 0},\veps') } \frac{\varrho(|\bw|)}{|\bw|}
        \Bigg( \int_{\bbR^d \setminus B({\bf 0},\veps)}  \frac{\varrho(|\bh|)}{|\bh|} \frac{\bh}{|\bh|} \times  \big( \bu(\bx+\bh+\bw)-\bu(\bx+\bw) \big)  \, \rmd \bh \\
            &\qquad \qquad - \int_{\bbR^d \setminus B({\bf 0},\veps)}  \frac{\varrho(|\bh|)}{|\bh|} \frac{\bh}{|\bh|} \times  \big( \bu(\bx+\bh)-\bu(\bx) \big)  \, \rmd \bh \Bigg) \cdot  \frac{\bw}{|\bw|} \, \rmd \bw = - \cD_{\varrho,\veps'} \circ \cC_{\varrho,\veps} \bu(\bx)\,.
    \end{split}
\end{equation*}}
\end{proof}
\begin{proof}[Proof of Proposition \ref{prop:DivOfCurl:SmoothFxns}]
Use Theorem \ref{thm:OperatorsWellDefdForSmoothFxns} to take the limit as $\veps$, $\veps' \to 0$ on both sides of \eqref{eq:DivOfCurl:Truncated}:
\begin{equation*}
    \cD_{\varrho} \circ \cC_{\varrho} \bu = - \cD_{\varrho} \circ \cC_{\varrho} \bu\,.
\end{equation*}
\end{proof}

\subsection{Curl of Curl Identity}

{We again proceed by computing the composition of the curl operator with itself in the truncated case and then using Theorem \ref{thm:OperatorsWellDefdForSmoothFxns} to prove that the same identity holds in the limit.}

\begin{proposition}\label{prop:CurlOfCurl:SmoothFxns}
The identity
\begin{equation}\label{eq:goal}
\cC_{\varrho} \circ \cC_{\varrho}\ub(\xb) = \cG_{\varrho} \circ \cD_{\varrho} \ub(\xb) - \cD_{\varrho} \circ \cG_{\varrho} \ub(\xb)
\end{equation}
holds for all $\bx \in \bbR^d$ if either
\begin{enumerate}
    \item[1)] $\bu \in C^2_b(\bbR^d;\bbR^d)$ with $d=3$, or
    \item[2)] $\varrho=\varrho_s$ and $\bu \in L^1_{2s}(\bbR^d;\bbR^d) \cap \scC^{2s+\sigma}(\bbR^d;\bbR^d)$ with $d=3$ and for $\sigma > 0$ sufficiently small.
\end{enumerate}
\end{proposition}

This can be shown immediately by applying \Cref{thm:OperatorsWellDefdForSmoothFxns} to the following version for the corresponding truncated operators.

\begin{theorem}\label{thm:CurlOfCurl:Truncated}
For any $\bu \in L^{\infty}(\bbR^d;\bbR^d)$ with $d=3$ and for any $\veps$, $\veps' > 0$
\begin{equation}\label{eq:CurlCurlId:Truncated}
    \cC_{\varrho,\veps} \circ \cC_{\varrho,\veps'}\ub(\xb) = \cG_{\varrho,\veps'} \circ \cD_{\varrho,\veps} \ub(\xb) - \cD_{\varrho,\veps} \circ \cG_{\varrho,\veps'} \ub(\xb)
\end{equation}
\end{theorem}

\begin{proof}
We require the following ``triple product'' identity,
\begin{equation}\label{eq:triple_product}
\ab \times (\bfb \times \cb) = (\ab \cdot \cb) \bfb - (\ab \cdot \bfb) \cb.  
\end{equation}

Unpacking the operator $\cC_{\varrho,\veps} \circ \cC_{\varrho,\veps'} \ub$ using the definition of $\cC_{\varrho,\veps}$ and changing coordinates, 
\begin{equation*}
\begin{split}
\cC_{\varrho,\veps} \circ \cC_{\varrho,\veps'}\ub(\xb) 
&=
\cC_{\varrho,\veps} (\cC_{\varrho,\veps'}\ub)(\xb) 
\\
&=
\int_{\bbR^d \setminus B({\bf 0},\veps)} \varrho (\bh) \frac{\bh}{|\bh|} \times \frac{(\cC_{\varrho,\veps'}\ub(\bx+\bh)-\cC_{\varrho,\veps'}\ub(\xb))}{|\bh|} \, \rmd \bh
\\
&=
\int_{\bbR^d \setminus B({\bf 0},\veps)} \frac{\varrho(\bh)}{|\bh|}  \frac{\bh}{|\bh|} \\
&\qquad 
\Bigg(
\int_{\bbR^d \setminus B({\bf 0},\veps')} \frac{\varrho(\bw)}{|\bw|}  \frac{\bw}{|\bw|} \times (\ub(\bx+\bh+\bw)-\ub(\bx+\bh)) \, \rmd \bw
\\
&\qquad \qquad \qquad \qquad - 
\int_{\bbR^d \setminus B({\bf 0},\veps')} \frac{\varrho(\bw)}{|\bw|}  \frac{\bw}{|\bw|} \times (\ub(\bx+\bw)-\ub(\xb)) \, \rmd\bw
\Bigg) \, \rmd\bh
\\
&=
 \int_{\bbR^d \setminus B({\bf 0},\veps)} \frac{\varrho(\bh)}{|\bh|}  \frac{\bh}{|\bh|} \times
 \int_{\bbR^d \setminus B({\bf 0},\veps')}
\frac{\varrho(\bw)}{|\bw|} \\
&\qquad \Bigg( \frac{\bw}{|\bw|} \times (\ub(\bx+\bh+\bw)-\ub(\bx+\bh))
-  \frac{\bw}{|\bw|} \times (\ub(\bx+\bw)-\ub(\xb))
\Bigg) \, \rmd \bw  \, \rmd \bh.
\end{split}
\end{equation*}

We are justified in using linearity of the integral in the last equality, since $\bw \mapsto \varrho(\bw) \frac{|\ub(\bx+\bw)-\ub(\bx)|}{|\bw|}$ is in $L^1(\bbR^d \setminus B({\bf 0},\veps))$ for any $\bu \in L^{\infty}(\bbR^d)$, for any $\veps > 0$ and for any $\bx \in \bbR^d$.

Some of these terms do not depend on $\bw$, so we can write
\begin{align*}
\cC_{\varrho,\veps} \circ \cC_{\varrho,\veps'}\ub(\xb)
&=
 \int_{\bbR^d \setminus B({\bf 0},\veps)} \frac{\varrho(\bh)}{|\bh|}
 \int_{\bbR^d \setminus B({\bf 0},\veps')}
\frac{\varrho(\bw)}{|\bw|}  \left[\frac{\bh}{|\bh|} \times \left( \frac{\bw}{|\bw|} \times (\ub(\bx+\bh+\bw)-\ub(\bx+\bh)) \right)\right. \\
        &\qquad \qquad  - \left.\frac{\bh}{|\bh|} \times
\left(  \frac{\bw}{|\bw|} \times (\ub(\bx+\bw)-\ub(\xb))
\right)\right] \, \rmd \bw  \, \rmd \bh.
\end{align*}

Now we use the identity \eqref{eq:triple_product} to write
\begin{align*}
\cC_{\varrho,\veps} \circ \cC_{\varrho,\veps'}\ub(\xb) 
=
\int_{\bbR^d \setminus B({\bf 0},\veps)} &  \frac{\varrho (\bh)}{|\bh|} \int_{\bbR^d \setminus B({\bf 0},\veps')} \frac{\varrho (\bw)}{|\bw|}  \Bigg[ 
\\
&\left( \frac{\bh}{|\bh|} \cdot (\ub(\bx+\bh+\bw)-\ub(\bx+\bh)) \right)
\frac{\bw}{|\bw|}  
\\
-
&\left( \frac{\bh}{|\bh|} \cdot \frac{\bw}{|\bw|}  \right)
(\ub(\bx+\bh+\bw)-\ub(\bx+\bh))
\\
-
&\left( \frac{\bh}{|\bh|} \cdot (\ub(\bx+\bw)-\ub(\bx)) \right)
\frac{\bw}{|\bw|}
\\
+
&\left( \frac{\bh}{|\bh|} \cdot \frac{\bw}{|\bw|}  \right)
(\ub(\bx+\bw)-\ub(\bx))
\Bigg] \, \rmd \bw  \, \rmd \bh
\\
=
\int_{\bbR^d \setminus B({\bf 0},\veps)} &  \frac{\varrho (\bh)}{|\bh|} \int_{\bbR^d \setminus B({\bf 0},\veps')}  \frac{\varrho (\bw)}{|\bw|}
\Bigg[
\\
&\left( \frac{\bh}{|\bh|} \cdot (\ub(\bx+\bh+\bw)-\ub(\bx+\bh)) \right)
\frac{\bw}{|\bw|}  
\\
-
&\left( \frac{\bh}{|\bh|} \cdot (\ub(\bx+\bw)-\ub(\bx)) \right)
\frac{\bw}{|\bw|}  
\\
-
&\left( \frac{\bh}{|\bh|} \cdot \frac{\bw}{|\bw|}  \right)
(\ub(\bx+\bh+\bw)-\ub(\bx+\bh))
\\
+
&\left( \frac{\bh}{|\bh|} \cdot \frac{\bw}{|\bw|}  \right)
(\ub(\bx+\bw)-\ub(\bx))
\Bigg] \, \rmd \bw  \, \rmd \bh\,.
\end{align*}
Since the last expression in the double integral is majorized by
\begin{equation}\label{eq:CurlOfCurl:Majorizer}
    C \chi_{ \{|\bh| \geq \veps\} } \chi_{ \{|\bw| \geq \veps'\} } \Vnorm{\bu}_{L^{\infty}(\bbR^d)} \frac{\varrho (\bh)}{|\bh|} \frac{\varrho (\bw)}{|\bw|} \in L^1( \bbR^d \times \bbR^d)
\end{equation}
we can use linearity of the double integral to separate the two former terms from the latter two. This gives
{\small \begin{align*}
\cC_{\varrho,\veps} \circ \cC_{\varrho,\veps'}\ub(\xb)
=
& \int_{\bbR^d \setminus B({\bf 0},\veps)} \int_{\bbR^d \setminus B({\bf 0},\veps')} \frac{\varrho (\bh)}{|\bh|} \frac{\varrho (\bw)}{|\bw|} \Bigg[
\\
& \qquad \left( \frac{\bh}{|\bh|} \cdot (\ub(\bx+\bh+\bw)-\ub(\bx+\bh)) \right)
\frac{\bw}{|\bw|}
-
\left( \frac{\bh}{|\bh|} \cdot (\ub(\bx+\bw)-\ub(\bx)) \right)
\frac{\bw}{|\bw|}  \Bigg] \, \rmd \bw \, \rmd \bh
\\
&\quad - \int_{\bbR^d \setminus B({\bf 0},\veps)} \int_{\bbR^d \setminus B({\bf 0},\veps')} \frac{\varrho (\bh)}{|\bh|} \frac{\varrho (\bw)}{|\bw|} \Bigg[ \\
&\qquad
\left( \frac{\bh}{|\bh|} \cdot \frac{\bw}{|\bw|}  \right)
(\ub(\bx+\bh+\bw)-\ub(\bx+\bh))
- \left( \frac{\bh}{|\bh|} \cdot \frac{\bw}{|\bw|}  \right)
(\ub(\bx+\bw)-\ub(\bx))
\Bigg] \, \rmd \bw  \, \rmd \bh \\
&:=
I \, (\text{first two lines above}) - II \, (\text{last two lines above}).
\end{align*} }
Now, we use the vector identity
\begin{equation*}
    (\ba \cdot \bfb) \bc = (\bc \otimes \ba) \bfb
\end{equation*}
and write
 {\small \begin{equation}\label{eq:CurlOfCurl:Pf1}
    \begin{split}
        II &= \int_{\bbR^d \setminus B({\bf 0},\veps)} \int_{\bbR^d \setminus B({\bf 0},\veps')} \frac{\varrho (\bh)}{|\bh|} \frac{\varrho (\bw)}{|\bw|} \Bigg[ \\
    & \qquad \left( (\ub(\bx+\bh+\bw)-\ub(\bx+\bh)) \otimes \frac{\bw}{|\bw|} \right) \frac{\bh}{|\bh|}
        +
        \left( (\ub(\bx+\bw)-\ub(\bx)) \otimes \frac{\bw}{|\bw|} \right) \frac{\bh}{|\bh|}  \Bigg] \, \rmd \bw \, \rmd \bh \\
    &= \int_{\bbR^d \setminus B({\bf 0},\veps)}  \frac{\varrho (\bh)}{|\bh|} \Bigg[
        \left( \int_{\bbR^d \setminus B({\bf 0},\veps')} \frac{\varrho (\bw)}{|\bw|} (\ub(\bx+\bh+\bw)-\ub(\bx+\bh)) \otimes \frac{\bw}{|\bw|} \, \rmd \bw \right) \frac{\bh}{|\bh|} \\
            &\qquad -
        \left( \int_{\bbR^d \setminus B({\bf 0},\veps')} \frac{\varrho (\bw)}{|\bw|}  (\ub(\bx+\bw)-\ub(\bx)) \otimes \frac{\bw}{|\bw|} \, \rmd \bw \right) \frac{\bh}{|\bh|}  \Bigg] \, \rmd \bh \\
    &= \int_{\bbR^d \setminus B({\bf 0},\veps)}  \frac{\varrho (\bh)}{|\bh|} \Bigg[
        \cG_{\varrho,\veps'} \bu(\bx+\bh) \frac{\bh}{|\bh|} -
        \cG_{\varrho,\veps'} \bu(\bx) \frac{\bh}{|\bh|}  \Bigg] \, \rmd \bh \\
        &= \cD_{\varrho,\veps} \circ \cG_{\varrho,\veps'} \bu(\bx)\,.
    \end{split}
\end{equation}}
Using linearity of the inner integral is again justified since the double integrand of $II$ is majorized by the function in \eqref{eq:CurlOfCurl:Majorizer}.

Last, the double integrand of $I$ is also majorized by the function in \eqref{eq:CurlOfCurl:Majorizer}. Therefore using Fubini's theorem and linearity of the integral 
{\small \begin{equation}\label{eq:CurlOfCurl:Pf2}
    \begin{split}
        I &=  \int_{\bbR^d \setminus B({\bf 0},\veps)} \int_{\bbR^d \setminus B({\bf 0},\veps')} \frac{\varrho (\bh)}{|\bh|} \frac{\varrho (\bw)}{|\bw|} 
        \Bigg[
            \frac{\bh}{|\bh|} \cdot (\ub(\bx+\bh+\bw)-\ub(\bx+\bh)-\bu(\bx+\bw) + \bu(\bx) )
        \Bigg] 
\frac{\bw}{|\bw|} \, \rmd \bw \, \rmd \bh \\
    &= \int_{\bbR^d \setminus B({\bf 0},\veps)} \int_{\bbR^d \setminus B({\bf 0},\veps')} \frac{\varrho (\bh)}{|\bh|} \frac{\varrho (\bw)}{|\bw|} 
        \Bigg[
            \left( \frac{\bh}{|\bh|} \cdot (\ub(\bx+\bh+\bw)-\ub(\bx+\bw)) \right) \\
            &\qquad \qquad \qquad - 
            \left( \frac{\bh}{|\bh|} \cdot (\bu(\bx+\bh) - \bu(\bx) ) \right)
        \Bigg] 
    \frac{\bw}{|\bw|} \, \rmd \bw \, \rmd \bh \\
    &= \int_{\bbR^d \setminus B({\bf 0},\veps')} \int_{\bbR^d \setminus B({\bf 0},\veps)} \frac{\varrho (\bw)}{|\bw|} 
        \Bigg[
            \frac{\varrho (\bh)}{|\bh|} \left( \frac{\bh}{|\bh|} \cdot (\ub(\bx+\bh+\bw)-\ub(\bx+\bw)) \right) \\
            &\qquad \qquad \qquad - 
            \frac{\varrho (\bh)}{|\bh|} \left( \frac{\bh}{|\bh|} \cdot (\bu(\bx+\bh) - \bu(\bx) ) \right)
        \Bigg] 
    \frac{\bw}{|\bw|} \, \rmd \bh \, \rmd \bw \\
    &= \int_{\bbR^d \setminus B({\bf 0},\veps')} \frac{\varrho (\bw)}{|\bw|} 
        \Bigg[
            \int_{\bbR^d \setminus B({\bf 0},\veps)}  \frac{\varrho (\bh)}{|\bh|} \left( \frac{\bh}{|\bh|} \cdot (\ub(\bx+\bh+\bw)-\ub(\bx+\bw)) \right) \, \rmd \bh \\
            &\qquad \qquad \qquad - \int_{\bbR^d \setminus B({\bf 0},\veps)} 
            \frac{\varrho (\bh)}{|\bh|} \left( \frac{\bh}{|\bh|} \cdot (\bu(\bx+\bh) - \bu(\bx) ) \right) \, \rmd \bh
        \Bigg] 
    \frac{\bw}{|\bw|} \, \rmd \bw \\
    &= \int_{\bbR^d \setminus B({\bf 0},\veps')} \frac{\varrho (\bw)}{|\bw|} 
        \Big[
            \cD_{\varrho,\veps} \bu(\bx+\bw) - \cD_{\varrho,\veps} \bu(\bx)
        \Big] 
    \frac{\bw}{|\bw|} \, \rmd \bw \\
    &= \cG_{\varrho,\veps'} \circ \cD_{\varrho,\veps} \bu(\bx)\,.
    \end{split}
\end{equation}}
Putting together \eqref{eq:CurlOfCurl:Pf1} and \eqref{eq:CurlOfCurl:Pf2} gives us the theorem.
\end{proof}
\begin{proof}[Proof of Proposition \ref{prop:CurlOfCurl:SmoothFxns}]
Follows from Theorem \ref{thm:OperatorsWellDefdForSmoothFxns} by passing to the limit as $\veps$, $\veps' \to 0$ in \eqref{eq:CurlCurlId:Truncated}.
\end{proof}


%

\section{Equivalence Kernel}\label{sec:eq-kernel}
In this section we rigorously show that for bounded $C^2$ functions, there exists an equivalence kernel for which the composition of the divergence and gradient operators corresponds to the (unweighted) nonlocal Laplace operator, i.e.  $-\cD_{\varrho} \circ \cG_{\varrho}=(-\Delta)_\varrho$. Furthermore, we use the kernel examples described in Section \ref{sec:OperatorDef} to illustrate our equivalence result. 
\begin{theorem}\label{thm:EquivalenceKernel}
Let $d$ and $N$ be positive integers. Suppose $\varrho$ is a radial kernel that satisfies \eqref{assumption:Kernel}, and suppose that
\begin{equation}\label{eq:KernelFullIntegrability}
    \frac{\varrho(|\bseta|)}{|\bseta|} \in L^1(\bbR^d)\,. \tag{K-INT}
\end{equation}
Then for functions $\bu \in C^2_b(\bbR^d;\bbR^N)$ the formula
\begin{equation}\label{eq:Diffusion}
    -\cD_{\varrho} \circ \cG_{\varrho} \bu(\bx) = \frac{1}{2}\intdm{\bbR^d}{ \varrho_{\text{eq}}(|\by|) \frac{(2\bu(\bx)-\bu(\bx+\by)-\bu(\bx-\by))}{|\by|^2} }{\by}
\end{equation}
holds, where the measurable function $\varrho_{\text{eq}}$ is defined as
\begin{equation}\label{eq:EquivalenceKernel}
    \varrho_{\text{eq}}(|\bseta|) := |\bseta|^{d} \intdm{\bbR^d}{ \frac{\varrho(|\bseta| |\bz|)}{|\bz|} \frac{\varrho(|\bseta| |\be_1 -\bz|)}{|\be_1-\bz|} { \frac{\be_1-\bz}{|\be_1 - \bz|} \cdot \frac{\bz}{|\bz|} } }{\bz}\,, \quad |\bseta| > 0\,.
\end{equation}
\end{theorem}

Henceforth we define the operator appearing in \eqref{eq:Diffusion} as
\begin{equation}\label{eq:DiffusionDefinition}
    (-\Delta)_{\varrho}\bu(\bx) := \frac{1}{2}\intdm{\bbR^d}{ \varrho_{\text{eq}}(|\by|) \frac{(2\bu(\bx)-\bu(\bx+\by)-\bu(\bx-\by))}{|\by|^2} }{\by}\,.
\end{equation}

\begin{proof}[Proof of Theorem \ref{thm:EquivalenceKernel}]
Unpacking the operator $\cD_{\varrho} \circ \cG_{\varrho} \bu$ and changing coordinates,
\begin{equation*}
    \begin{split}
        &\cD_{\varrho} \circ \cG_{\varrho} \bu(\bx) \\
            &= \intdm{\bbR^d }{\varrho(|\bh|) \frac{\cG_{\varrho} \bu(\bx+\bh)-\cG_{\varrho} \bu(\xb)}{|\bh|} \frac{\bh}{|\bh|}  }{\bh} \\
            &= \int_{\bbR^d}  \frac{\varrho(|\bh|)}{|\bh|} \Bigg(\int_{\bbR^d} \frac{\varrho(|\bw|)}{|\bw|}  \Big( \bu(\bx+\bh+\bw) - \bu(\bx+\bh) \Big) \otimes \frac{\bw}{|\bw|} \, \rmd \bw \\
                &\qquad -\int_{\bbR^d} \frac{\varrho(|\bw|)}{|\bw|} \Big(\bu(\bx+\bw) - \bu(\bx) \Big) \otimes \frac{\bw}{|\bw|} \, \rmd \bw \Bigg) \frac{\bh}{|\bh|} \, \rmd \bh \\
            &= \int_{\bbR^d}  \frac{\varrho(|\bh|)}{|\bh|} \Bigg( 
                 \int_{\bbR^d} \frac{\varrho(|\bw|)}{|\bw|}  \big(\bu(\bx+\bh+\bw) - \bu(\bx+\bh) -\bu(\bx+\bw) + \bu(\bx) \big) \otimes \frac{\bw}{|\bw|} \, \rmd \bw \Bigg) \frac{\bh}{|\bh|} \, \rmd \bh\,.
    \end{split}
\end{equation*}
We are justified in using linearity of the integral in the last equality, since by \eqref{eq:KernelFullIntegrability} $\bw \mapsto \varrho(\bw) \frac{|\ub(\bx+\bw)-\ub(\bx)|}{|\bw|} \in L^1(\bbR^d)$ for any $\bu \in L^{\infty}(\bbR^d)$, for any $\veps > 0$ and for any $\bx \in \bbR^d$. Using the vector identity $(\ba \otimes \bfb) \bc = (\bfb \cdot \bc) \ba$ brings us to
\begin{equation*}
    \cD_{\varrho} \circ \cG_{\varrho} \bu(\bx) = \int_{\bbR^d} \int_{\bbR^d} \frac{\varrho(|\bh|)}{|\bh|}  \frac{\varrho(|\bw|)}{|\bw|} \Bigg( \frac{\bw}{|\bw|} \cdot \frac{\bh}{|\bh|} \Bigg) \big(\bu(\bx+\bh+\bw) - \bu(\bx+\bh) - \bu(\bx+\bw) + \bu(\bx) \big)  \, \rmd \bw \, \rmd \bh\,.
\end{equation*}
The expression in the double integral is majorized by
\begin{equation}\label{eq:DivOfGrad:Majorizer}
    C \Vnorm{\bu}_{L^{\infty}(\bbR^d)} \frac{\varrho (|\bh|)}{|\bh|} \frac{\varrho (|\bw|)}{|\bw|}\,,
\end{equation}
which belongs to $L^1( \bbR^d \times \bbR^d)$ by Tonelli's theorem.
Therefore, Fubini's theorem is justified in the following splitting of the integrand:
\begin{equation*}
\begin{split}
    \cD_{\varrho} \circ \cG_{\varrho} \bu(\bx) &= \int_{\bbR^d} \int_{\bbR^d} \frac{\varrho(|\bh|)}{|\bh|}  \frac{\varrho(|\bw|)}{|\bw|} \Bigg( \frac{\bw}{|\bw|} \cdot \frac{\bh}{|\bh|} \Bigg) \bu(\bx+\bh+\bw) \, \rmd \bw \, \rmd \bh \\
        &\qquad - \int_{\bbR^d} \frac{\varrho(|\bh|)}{|\bh|} \left( \left[ \int_{\bbR^d}   \frac{\varrho(|\bw|)}{|\bw|} \frac{\bw}{|\bw|} \, \rmd \bw \right] \cdot \frac{\bh}{|\bh|} \right) \bu(\bx+\bh) \, \rmd \bh \\
        &\qquad - \int_{\bbR^d} \frac{\varrho(|\bw|)}{|\bw|} \left( \left[ \int_{\bbR^d}   \frac{\varrho(|\bh|)}{|\bh|} \frac{\bh}{|\bh|} \, \rmd \bh \right] \cdot \frac{\bw}{|\bw|} \right) \bu(\bx+\bw) \, \rmd \bw \\
        &\qquad + \int_{\bbR^d} \int_{\bbR^d} \frac{\varrho(|\bh|)}{|\bh|}  \frac{\varrho(|\bw|)}{|\bw|} \Bigg( \frac{\bw}{|\bw|} \cdot \frac{\bh}{|\bh|} \Bigg) \, \rmd \bw \, \rmd \bh \, \bu(\bx) \,.
\end{split}
\end{equation*}
The inner integrals on the second and third lines are both zero, since the respective integrands are odd. The last line is zero for the same reason. Therefore, we can subtract any multiple of the last line from $\cD_{\varrho} \circ \cG_{\varrho} \bu(\bx)$.
Combining this fact, along with splitting the integral and changing coordinates, gives
\begin{equation*}
\begin{split}
    \cD_{\varrho} \circ \cG_{\varrho} \bu(\bx) &= \frac{1}{2} \int_{\bbR^d} \int_{\bbR^d} \frac{\varrho(|\bh|)}{|\bh|}  \frac{\varrho(|\bw|)}{|\bw|} \Bigg( \frac{\bw}{|\bw|} \cdot \frac{\bh}{|\bh|} \Bigg) \bu(\bx+\bh+\bw) \, \rmd \bw \, \rmd \bh \\
        &\qquad +\frac{1}{2} \int_{\bbR^d} \int_{\bbR^d} \frac{\varrho(|\bh|)}{|\bh|}  \frac{\varrho(|\bw|)}{|\bw|} \Bigg( \frac{\bw}{|\bw|} \cdot \frac{\bh}{|\bh|} \Bigg) \bu(\bx+\bh+\bw) \, \rmd \bw \, \rmd \bh \\
    &= \frac{1}{2} \int_{\bbR^d} \int_{\bbR^d} \frac{\varrho(|\bh|)}{|\bh|}  \frac{\varrho(|\bw|)}{|\bw|} \Bigg( \frac{\bw}{|\bw|} \cdot \frac{\bh}{|\bh|} \Bigg) \bu(\bx+\bh+\bw) \, \rmd \bw \, \rmd \bh \\
        &\qquad +\frac{1}{2} \int_{\bbR^d} \int_{\bbR^d} \frac{\varrho(|\bh|)}{|\bh|}  \frac{\varrho(|\bw|)}{|\bw|} \Bigg( \frac{\bw}{|\bw|} \cdot \frac{\bh}{|\bh|} \Bigg) \bu(\bx-\bh-\bw) \, \rmd \bw \, \rmd \bh \\
        &= \frac{1}{2} \int_{\bbR^d} \int_{\bbR^d} \frac{\varrho(|\bh|)}{|\bh|}  \frac{\varrho(|\bw|)}{|\bw|} \Bigg( \frac{\bw}{|\bw|} \cdot \frac{\bh}{|\bh|} \Bigg) \\
        &\qquad \qquad \cdot \big( \bu(\bx+\bh+\bw) + \bu(\bx-\bh-\bw) - 2 \bu(\bx) \big)\, \rmd \bw \, \rmd \bh\,.
    \end{split}
\end{equation*}
Now, we iterate the integrals and introduce the coordinate change $\by = \bw+\bh$:
\begin{equation*}
\begin{split}
    \cD_{\varrho} \circ \cG_{\varrho} \bu(\bx) &= \frac{1}{2} \int_{\bbR^d} \int_{\bbR^d} \frac{\varrho(|\bh|)}{|\bh|}  \frac{\varrho(|\by-\bh|)}{|\by-\bh|} { \frac{\by-\bh}{|\by-\bh|} \cdot \frac{\bh}{|\bh|} } \big( \bu(\bx+\by) + \bu(\bx-\by) - 2 \bu(\bx) \big) \, \rmd \by \, \rmd \bh\,.
    \end{split}
\end{equation*}
We can interchange the order of integration, since the integrand remains majorized by \eqref{eq:DivOfGrad:Majorizer}. So we have
\begin{equation*}
    -\cD_{\varrho} \circ \cG_{\varrho} \bu(\bx) = \frac{1}{2} \int_{\bbR^d} \varrho_{eq}(\by) \frac{ \big( 2 \bu(\bx) - \bu(\bx+\by) - \bu(\bx-\by)  \big) }{|\by|^2} \, \rmd \by \,,
\end{equation*}
where
\begin{equation*}
    \varrho_{\text{eq}}(\by) = |\by|^2 \intdm{\bbR^d}{ \frac{\varrho(|\bh|)}{|\bh|} \frac{\varrho(|\by-\bh|)}{|\by-\bh|} { \frac{\by-\bh}{|\by-\bh|} \cdot \frac{\bh}{|\bh|} } }{\bh}\,.
\end{equation*}
In order to conclude with the formula \eqref{eq:EquivalenceKernel} we will show that $\varrho_{\text{eq}}$ actually only depends on $|\by|$. For any $\by \neq {\bf 0}$, let $\bw = \frac{\bh}{|\by|}$ and change coordinates:
\begin{equation*}
    \varrho_{\text{eq}}(\by) := |\by|^{d} \intdm{\bbR^d}{ \frac{\varrho(|\by| |\bw|)}{|\bw|} \frac{\varrho(|\by| |\frac{\by}{|\by|} -\bw|)}{|\frac{\by}{|\by|}-\bw|} { \frac{\frac{\by}{|\by|}-\bw}{|\frac{\by}{|\by|}-\bw|} \cdot \frac{\bw}{|\bw|} } }{\bw}\,.
\end{equation*}
Let $\bR(\by)$ be the rotation such that $\bR \frac{\by}{|\by|} =\be_1$, where $\be_1 = (1,0,\ldots, 0)$. Then letting $\bz = \bR \bw$ and changing coordinates gives
\begin{equation*}
    \varrho_{\text{eq}}(\by) = |\by|^{d} \intdm{\bbR^d}{ \frac{\varrho(|\by| |\bz|)}{|\bz|} \frac{\varrho(|\by| |\be_1 -\bz|)}{|\be_1-\bz|} { \frac{\be_1-\bz}{|\be_1 - \bz|} \cdot \frac{\bz}{|\bz|} } }{\bz}\,.
\end{equation*}
\end{proof}

The previous theorem relies heavily on the assumption \eqref{eq:KernelFullIntegrability}, which does not hold for singular kernels such as $\varrho_s$. Nevertheless, a pointwise equivalence kernel can be defined, as we show in the next lemma.

\begin{lemma}\label{lma:EquivalenceKernel:Singular}
    Suppose that a radial kernel $\varrho$ satisfies \eqref{assumption:Kernel}. Assume that $\varrho$ satisfies the following conditions.
    Define the function
\begin{equation*}
    \Psi(r) := \int_r^{\infty} \frac{\varrho(\theta)}{\theta} \, \rmd \theta < \infty\,, \qquad r > 0\,.
\end{equation*}
Note that $\Psi: (0,\infty) \to (0,\infty)$ is well-defined by assumption.
Suppose that $\Psi$ satisfies
\begin{equation}\label{assumption:EquivalenceKernel}
    r^{d-1} \Psi(r) \in L^1_{\text{loc}}([0,\infty))\,, \qquad \Psi \in C^2((0,\infty))\,. \tag{K-EQ}
\end{equation}
    Then a pointwise equivalence kernel can be defined in the following way:
    \begin{equation*}
    \varrho_{\text{eq}}(|\bseta|) := \lim\limits_{\veps,\veps' \to 0} \varrho_{\text{eq},\veps,\veps'}(|\bseta|)\,,
\end{equation*}
for any $|\bseta| > 0\,,$ 
where the measurable function $\varrho_{\text{eq},\veps,\veps'}$ is defined for $\veps > 0$ and $\veps' > 0$ as
\begin{equation}\label{eq:pvEquivalenceKernel}
    \varrho_{\text{eq},\veps,\veps'}(|\bseta|) := |\bseta|^{d} \intdm{\bbR^d}{ \chi_{\{|\be_1-\bz| > \veps' \}} \chi_{\{|\bz| > \veps \}} \frac{\varrho(|\bseta| |\bz|)}{|\bz|} \frac{\varrho(|\bseta| |\be_1 -\bz|)}{|\be_1-\bz|} { \frac{\be_1-\bz}{|\be_1 - \bz|} \cdot \frac{\bz}{|\bz|} } }{\bz}.
\end{equation}
\end{lemma}

\begin{proof}
To begin, we split the integral. For any $\veps > 0$, define the sets
\begin{equation*}
    A_{\veps,1} := \{ \bz \, : \, \veps \leq |\bz| \leq \frac{1}{2} \}\,, \quad A_{\veps,2} := \{ \bz \, : \, \veps \leq |\be_1-\bz| \leq \frac{1}{2} \}\,, \quad A_{1} := \{ \bz \, : \, \frac{1}{2} \leq |\bz| \text{ and } \frac{1}{2} \leq |\be_1-\bz| \}\,.
\end{equation*}
Then
\begin{equation*}
    \varrho_{\text{eq},\veps,\veps'}(|\bseta|) = \int_{A_{\veps,1}} \cdots +  \int_{A_{\veps',2}} \cdots +  \int_{A_{1}} \cdots
\end{equation*}
Clearly the third integral is an absolutely convergent integral. 
Letting $\by = \be_1 - \bz$, a change of coordinates gives
\begin{equation*}
    \begin{split}
    &\intdm{A_{\veps',2}}{ \frac{\varrho(|\bseta| \, |\be_1-\bz|)}{|\be_1-\bz|} \frac{\varrho(|\bseta| \, |\bz|)}{|\bz|} { \frac{\be_1-\bz}{|\be_1-\bz|} \cdot \frac{\bz}{|\bz|} }  }{\bz} \\
        &\quad =  \intdm{\bbR^d}{ \chi_{ \{ \frac{1}{2} \geq |\be_1-\bz| \geq \veps' \} } \cdot  \frac{\varrho(|\bseta| \, |\be_1-\bz|)}{|\be_1-\bz|} \frac{\varrho(|\bseta| \, |\bz|)}{|\bz|} { \frac{\be_1-\bz}{|\be_1-\bz|} \cdot \frac{\bz}{|\bz|} } 
        }{\bz} \\
        &\quad =  \intdm{\bbR^d}{ \chi_{ \{ \frac{1}{2} \geq |\by| \geq \veps' \} } \cdot  \frac{\varrho(|\bseta| \, |\by|)}{|\by|} \frac{\varrho(|\bseta| \, |\be_1-\by|)}{|\be_1-\by|} { \frac{\by}{|\by|} \cdot \frac{\be_1-\by}{|\be_1-\by|} } 
        }{\by} \\
        &\quad = \intdm{A_{\veps',1}}{ \frac{\varrho(|\bseta| \, |\be_1-\bz|)}{|\be_1-\bz|} \frac{\varrho(|\bseta| \, |\bz|)}{|\bz|} { \frac{\be_1-\bz}{|\be_1-\bz|} \cdot \frac{\bz}{|\bz|} }   }{\bz}\,.
    \end{split}
\end{equation*}
Thus it suffices to show that the quantity
\begin{align}\label{eq:EquivalenceKernel:Proof1}
\begin{split}
    \sup_{\veps > 0} \wt{\varrho}_{\text{eq},\veps,1}(|\bseta|) &:= \intdm{\bbR^d}{ \chi_{ \{ \frac{1}{2} \geq |\bz| \geq \veps \} } \cdot  \frac{\varrho(|\bseta| \, |\be_1-\bz|)}{|\be_1-\bz|} \frac{\varrho(|\bseta| \, |\bz|)}{|\bz|} { \frac{\be_1-\bz}{|\be_1-\bz|} \cdot \frac{\bz}{|\bz|} }    }{\bz} \\ &< \infty
\end{split}
\end{align}
for any fixed $|\bseta| > 0$. We assume $\veps < 1/4$ from here on.

Note for any fixed $\delta > 0$ and for $\ba \in \{ {\bf 0}, \be_1 \}$
\begin{equation*}
    \grad_{\bz} \Psi( \delta |\ba-\bz| ) = \frac{\varrho(\delta |\ba-\bz|)}{\delta |\ba-\bz|} \cdot \delta \frac{\ba-\bz}{|\ba-\bz|}\,;
\end{equation*}
Thus
\begin{equation*}
    \begin{split}
        \wt{\varrho}_{\text{eq},\veps,1}(\delta) &= \intdm{\{ \frac{1}{2} \geq |\bz| \geq \veps \}}{ { \grad_{\bz} \Psi( \delta |\be_1-\bz| ) \,\cdot\, \grad_{\bz} \Psi( \delta |\bz| ) }  }{\bz} \\
        &= \intdm{\{ \frac{1}{2} \geq |\bz| \geq \veps \}}{  \Delta_{\bz} \Psi( \delta |\be_1-\bz| )  \,  \Psi( \delta |\bz| ) }{\bz} 
        + \intdm{ \{|\bz| = \veps\} }{ \grad_{\bz} \Psi(\delta|\be_1-\bz|) \cdot \frac{\bz}{|\bz|} \, \Psi(\delta |\bz|) }{\sigma(\bz)} \\
            &\quad + \intdm{ \{|\bz| = 1/2\} }{ \grad_{\bz} \Psi(\delta|\be_1-\bz|) \cdot \frac{\bz}{|\bz|} \, \Psi(\delta |\bz|) }{\sigma(\bz)}\,.
    \end{split}
\end{equation*}
Note that $\Psi \in C^2((0,\infty))$ and the argument $\delta |\be_1 - \bz|$ lives in a bounded set far away from $0$. Note also that $\Psi(\delta |\bz|) \in L^1_{\text{loc}}(\bbR^d)$ by assumption.
Therefore the first and third integrals are both finite and bounded uniformly in $\veps$.
As for the second integral, a change of variables gives
\begin{multline}
        C  \intdm{ \{|\bx| = \veps\} }{ \grad_{\bz} \Psi(\delta|\be_1-\bz|) \cdot \frac{\bz}{|\bz|} \, \Psi(\delta |\bz|) }{\sigma(\bz)} \\
         = C\intdm{\bbS^{d-1}}{ \veps^{d-1} \Psi(\delta \veps) \frac{\varrho(\delta |\be_1-\veps \bw|)}{|\be_1-\veps \bw|} { \frac{\be_1-\veps \bw}{|\be_1 - \veps \bw|} \cdot \bw } }{\sigma(\bw)} := I\,.
\end{multline}
Now for $\bw \in \bbS^{d-1}$ and for  $\veps \in [0,1/4)$ define the function $h_{\bw,\delta}(\veps) := \Psi(\delta |\be_1 - \veps \bw|)$. For any choice of $\bw$, we have that $\veps \mapsto h_{\bw,\delta}(\veps)$ is $C^2$ and its derivatives are uniformly bounded (the bound is also uniform with respect to $\bw$). Thus we can write $I$ as
\begin{equation*}
    I = C\intdm{\bbS^{d-1}}{ \veps^d \Psi(\delta \veps) \frac{h_{\bw,\delta}'(\veps)}{\veps}  }{\sigma(\bw)}\,.
\end{equation*}
Note that $h_{\bw,\delta}'(0) = \varrho(\delta) { \be_1 \cdot \bw }$ and thus $\intdm{\bbS^{d-1}}{h_{\bw,\delta}'(0)}{\sigma(\bw)} = 0$. We then see by applying the mean value theorem that
\begin{equation*}
    I = C\intdm{\bbS^{d-1}}{ \veps^d \Psi(\delta \veps) \frac{h_{\bw,\delta}'(\veps)-h_{\bw,\delta}'(0)}{\veps}  }{\sigma(\bw)} = O \left( \veps^d \Psi(\delta \veps) \right)\,.
\end{equation*}
We claim that $\lim\limits_{\veps \to 0} \veps^d \Psi(\delta \veps) = 0$, and so \eqref{eq:EquivalenceKernel:Proof1} will follow. To see this claim, note that for any nonincreasing function $f \in L^1_{\text{loc}}([0,\infty)) \cap C^0((0,\infty))$
\begin{equation*}
    \frac{x f(x)}{2} \leq \int_{x/2}^{x} f(y) \, \rmd y \to 0 \text{ as } x \to 0\,,
\end{equation*}
by continuity of the integral.
\end{proof}

By the calculations in the proof of Theorem \ref{thm:EquivalenceKernel} it follows that
\begin{equation}\label{eq:pvDiffusion}
    \cD_{\varrho,\veps} \circ \cG_{\varrho,\veps'} \bu(\bx) = \frac{1}{2} \intdm{\bbR^d}{ \varrho_{\text{eq},\veps,\veps'}(|\by|) \frac{2 \bu(\bx) - \bu(\bx+\by) -\bu(\bx-\by) }{|\by|^2} }{\by} \text{ for } \veps\,, \veps' > 0\,.
\end{equation}
Unfortunately it is unclear if the limit as $\veps$, $\veps' \to 0$ can be taken for general kernels satisfying \eqref{assumption:EquivalenceKernel}, even if $\bu$ is smooth. However, for specific examples of $\varrho$ we can show that the integrand on the right-hand side of \eqref{eq:pvDiffusion} is bounded by an $L^1$ function uniformly in $\veps$ and $\veps'$. Then the limit can be taken on both sides of \eqref{eq:pvDiffusion} by \Cref{thm:OperatorsWellDefdForSmoothFxns} and by the Lebesgue Dominated Convergence theorem to conclude that formula \eqref{eq:Diffusion} holds for any $\bu \in C^2_b(\bbR^d)$.
The following examples illustrate the situation.

\begin{example}\label{example:equivalence:FractionalKernel}
Direct calculation shows that the fractional kernel $\varrho_s(|\bseta|)$ satisfies the conditions of \Cref{lma:EquivalenceKernel:Singular}. Moreover, \eqref{eq:pvEquivalenceKernel} for this particular kernel becomes  
   $ \varrho_{s,\text{eq},\veps,\veps'}(|\bseta|) = \frac{C_{d,s,\veps,\veps'}}{|\bseta|^{d+2s-2}}\,,$
where the sequence of constants $C_{d,s,\veps,\veps'}$ is given by 
\[
C_{d,s,\veps,\veps'} := \intdm{\bbR^d}{ \chi_{\{|\be_1-\bz| > \veps' \}} \chi_{\{|\bz| > \veps \}} \frac{1}{|\bz|^{d+2s-2}} \frac{1}{|\be_1-\bz|^{d+2s-2}} { \frac{\be_1-\bz}{|\be_1 - \bz|} \cdot \frac{\bz}{|\bz|} } }{\bz}.
\]
By the same line of reasoning as in the proof of \Cref{lma:EquivalenceKernel:Singular} we see that the constants $C_{d,s,\veps,\veps'}$ converge to a constant $C_{d,s}$ as $\veps$, $\veps' \to 0$. Using the Fourier transform (see \cite{d2020unified}) it follows that $C_{d,s} := \frac{2^{2s} s \Gamma(\frac{d}{2}+s)}{\pi^{d/2} \Gamma(1-s)}$.
We can therefore conclude that \eqref{eq:Diffusion} holds. We summarize this result in the following proposition: 
\begin{proposition}\label{prop:FractionalLaplaceIsDivGrad}
Let $s \in (0,1)$. Suppose that either $\bu \in C^2_b(\bbR^d;\bbR^N)$, or $\bu \in L^1_{2s}(\bbR^d;\bbR^N) \cap \scC^{2s+\sigma}(\bbR^d;\bbR^N)$ for some $\sigma > 0$ small. Then the function $(-\Delta)_{\varrho_s} \bu(\bx)$ defined in \eqref{eq:DiffusionDefinition} coincides with the \textit{fractional Laplacian}
\begin{equation*}
    (-\Delta)^s \bu(\bx) := C_{d,s} \int_{\bbR^d} \frac{2 \bu(\bx)-\bu(\bx+\by) - \bu(\bx-\by)}{|\by|^{d+2s}} \, \rmd \bh\,, \qquad \bx \in \bbR^d\,.
\end{equation*}
Put another way,
\begin{equation*}
    -\cD_s \circ \cG_s \bu(\bx) = (-\Delta)^s \bu(\bx) \text{ for every } \bx \in \bbR^d\,.
\end{equation*}
\end{proposition}

\begin{proof}
    If $\bu$ is in either set of function spaces, the limit as $\veps,\veps' \to 0$ can be taken on the left-hand side of \eqref{eq:pvDiffusion} by \Cref{thm:OperatorsWellDefdForSmoothFxns}. The limit on the right-hand side will follow by the Lebesgue Dominated Convergence theorem. First note that $C_{d,s,\veps,\veps'}$ is bounded by some constant $\wt{C}(d,s)$.
    Then the integrand is majorized by
    \begin{equation*}
        4 \wt{C}(d,s) \chi_{ \{|\by| < 1 \} } \frac{ \sum_{|\gamma| = 2} \Vnorm{D^{\gamma} \bu}_{L^{\infty}(\bbR^d)} }{|\by|^{d+2s-2}} 
            + 4 \wt{C}(d,s) \chi_{ \{|\by| \geq 1 \} } \frac{ \Vnorm{\bu}_{L^{\infty}(\bbR^d)} }{|\by|^{d+2s}}
    \end{equation*}
    in case 1, or by
    \begin{equation*}
        2 \wt{C}(d,s) \chi_{ \{|\by| < 1 \} } \frac{ [\bu]_{C^{0,2s+\sigma}(\bbR^d)} }{|\by|^{d-\sigma}}
            + 4 \wt{C}(d,s) \chi_{ \{|\by| \geq 1 \} } \frac{ \Vnorm{\bu}_{L^{\infty}(\bbR^d)} }{|\by|^{d+2s}}
    \end{equation*}
    in case 2 with $s< 1/2$. In case 2 with $s \geq 1/2$ we have the bound
    \begin{equation*}
    \begin{split}
        \left| \frac{2 \bu(\bx) - \bu(\bx+\by) - \bu(\bx-\by)}{|\by|^{d+2s}} \right|
        =& \left|
        \chi_{ \{|\by| < 1 \} } \frac{ \int_0^1 \big( \grad \bu(\bx-t\by) - \grad \bu(\bx+t\by) \big) \by \, \rmd t }{|\by|^{d+2s}} \right.\\ 
            &\left. + \chi_{ \{|\by| \geq 1 \} } \frac{2 \bu(\bx) - \bu(\bx+\by) - \bu(\bx-\by)}{|\by|^{d+2s}} \right| \\
        \leq&
        \chi_{ \{|\by| < 1 \} } \frac{ [\grad \bu]_{C^{0,2s+\sigma-1}(\bbR^d)} }{|\by|^{d-\sigma}}
            + 4 \chi_{ \{|\by| \geq 1 \} } \frac{ \Vnorm{\bu}_{L^{\infty}(\bbR^d)} }{|\by|^{d+2s}}\,.
    \end{split}
    \end{equation*}
    In all cases the bounding function is in $L^1(\bbR^d)$, and the proof is complete.
\end{proof}

\end{example}


\begin{example}
The truncated fractional kernel $\varrho_{s,\delta}(|\bseta|)$ does not satisfy the conditions of \Cref{lma:EquivalenceKernel:Singular}. Nevertheless, when $d=1$ the formula \eqref{eq:pvEquivalenceKernel} holds for almost every $\eta \neq 0$. This can be seen directly by computing the equivalence kernel:
\begin{equation*}
    \varrho_{s,\delta,\text{eq},\veps,\veps'}(|\eta|) = \frac{(c_{1,s})^2}{|\eta|^{1+2s-2}} \intdm{\bbR}{ 
    \chi_{ \{ \veps< |z| < \frac{\delta}{|\eta|} \} }  \chi_{ \{ \veps' < |z-1| < \frac{\delta}{|\eta|} \} }
    \frac{ 1-z}{|1-z|^{2+s}} \cdot \frac{z}{|z|^{2+s}}  }{z}\,, \qquad \eta \neq 0\,.
\end{equation*}
The integral can be computed explicitly. Let $_2 F_1(a,b;c;z)$ denote the \textit{hypergeometric function}; see \cite[Equation 15.1.1]{abramowitz1988handbook} and \Cref{apdx:HyperGeometricFxn} for the definition.

The derivative identity \eqref{eq:HypergeoDerivFormula1} implies that the function
\begin{equation*}
    F_s(x) :=
        \begin{cases}
            - \frac{1}{s(-x)^s} \text{}_2 F_1(-s,1+s;1-s;x)\,, & \quad x \leq 0\,, \\
            \frac{1}{s(1-x)^s} \text{}_2 F_1(-s,1+s;1-s;1-x)\,, & \quad 0 < x < 1\,, \\
            \frac{1}{s(x-1)^s} \text{}_2 F_1(-s,1+s;1-s;1-x)\,, & \quad x < 1\,, \\
        \end{cases}
\end{equation*}
satisfies $F_s'(x) = \frac{ 1-x}{|1-x|^{2+s}} \cdot \frac{x}{|x|^{2+s}} $ for all $x \in \bbR \setminus \{0,1\}$. Therefore,
{\small \begin{align*}
   & \varrho_{s,\delta,\text{eq},\veps,\veps'}(|\eta|)\\ &= \frac{(c_{1,s})^2}{|\eta|^{1+2s-2}}
    \begin{cases}
        F_s(-\veps) - F_s(1-\frac{\delta}{|\eta|}) + F_s(1-\veps') - F_s(\veps) + F_s(\frac{\delta}{|\eta|}) - F_s(1+\veps')\,, & \quad \frac{\delta}{|\eta|} > 1\,, \\
        F_s(\frac{\delta}{|\eta|}) - F_s(1-\frac{\delta}{|\eta|})\,, & \quad \frac{1}{2} < \frac{\delta}{|\eta|} < 1\,, \\
        0\,, & \quad \frac{1}{2} \geq \frac{\delta}{|\eta|}\,.
    \end{cases}
\end{align*} }
Now we compute the limit as $\veps$, $\veps' \to 0$. To do this, we need the following limits for $F_s(z)$.

\begin{theorem}\label{thm:PVofFs}
    For $s \in (0,1)$,
    \begin{equation}\label{eq:AntiDer:Limitat0}
        \lim\limits_{\veps \to 0} F_s(\veps) - F_s(-\veps) = \kappa_s :=  \begin{cases}
        \frac{\Gamma(1-s) \Gamma(-s)}{s \Gamma(-2s)}\,, &\quad \text{ if } s \neq 1/2\,, \\
        0 &\quad \text{ if } s = 1/2\,,
    \end{cases}
    \end{equation}
    and
    \begin{equation}\label{eq:AntiDer:Limitat1}
        \lim\limits_{\veps \to 0} F_s(1+\veps) - F_s(1-\veps) = 0\,.
    \end{equation}
\end{theorem}
See \Cref{apdx:HyperGeometricFxn} for the proof. An immediate corollary is the explicit formula for $\varrho_{s,\delta,\text{eq}}$.


\begin{corollary}
For all $|\eta| \neq 0$ and $|\eta| \neq \delta$
\begin{equation}\label{eq:Truncated:DefnOfEquivKernel}
    \varrho_{s,\delta,\text{eq}}(|\eta|) =
    \lim\limits_{\veps \to 0} \varrho_{s,\delta,\text{eq},\veps,\veps'}(|\eta|) = 
    \frac{(c_{1,s})^2 }{|\eta|^{1+2s-2}} G_{s} 
\left( \frac{\delta}{|\eta|} \right)  \,,
\end{equation}
where $G_s : (0,\infty) \setminus \{1\} \to \bbR$ is defined as
\begin{equation}\label{eq:DefnOfRemainder}
    G_s(x) :=
    \begin{cases}
        0\,, & \quad 0 < x \leq \frac{1}{2}\,, \\
        F_s(x) - F_s(1-x)\,, & \quad \frac{1}{2} < x < 1 \,, \\
        F_s(x) - F_s(1-x) - \kappa_s \,, & \quad 1 < x \,. \\
    \end{cases}
\end{equation}
\end{corollary}

We now investigate properties of $\varrho_{s,\delta,\text{eq}}$ that are desirable for applications. 
To do this we need the following results concerning $G_s$.

\begin{theorem}\label{thm:LimitsOfGs}
    For every $s \in (0,1)$ and for every $\tau > 0$, the function $G_s \big|_{(0,\infty) \setminus (1-\tau,1+\tau) }$ is continuous and bounded. Moreover,
    \begin{equation}\label{eq:LimOfGAtInfty}
        \lim\limits_{x \to \infty} G_s(x) = \frac{2\Gamma(1-s) \Gamma(1+2s)}{s\Gamma(1+s) } - \kappa_s
    \end{equation}
    and
    \begin{equation}\label{eq:LimOfGAt1}
        \lim\limits_{x \to 1} |x-1|^s G_s(x) = \frac{2}{s}\,.
    \end{equation}
\end{theorem}
See \Cref{apdx:HyperGeometricFxn} for the proof.

\begin{theorem}[Properties of $\varrho_{s,\delta,\text{eq}}$]
    Let $\delta > 0$ and $s \in (0,1)$. Then $\varrho_{s,\delta,\text{eq}}$ is finite and differentiable for all $\eta \in \bbR \setminus \{-\delta,0,\delta \}$. At $|\eta| = \delta$ the function has a singularity of order $s$; that is,
    \begin{equation}\label{eq:TruncKernel:LimAt1}
        \lim\limits_{|\eta| \to \delta} |\eta|^{1+2s-2} \left| \frac{\delta}{|\eta|} -1 \right|^{s} \varrho_{s,\delta,\text{eq}}(|\eta|) = \frac{2 (c_{1,s})^2}{s}.
    \end{equation}
    Additionally, $\varrho_{s,\delta,\text{eq}}$ is compactly supported with $\supp \varrho_{s,\delta,\text{eq}} = \overline{B(0,2\delta)}$,
    \begin{equation}\label{eq:TruncEquivKernel:Nonnegative}
        \varrho_{s,\delta,\text{eq}}(|\eta|) \geq 0 \text{ for all } \eta \in \bbR \setminus \{ -\delta, 0, \delta \}\,,
    \end{equation}
    and
    \begin{equation}\label{eq:TruncEquivKernel:Integrable}
        \varrho_{s,\delta,\text{eq}} \in L^1(\bbR^d)\,.
    \end{equation}
    Moreover, $\varrho_{s,\delta,\text{eq}}$ is consistent with $\varrho_{s,\text{eq}}$; that is, for every fixed $|\eta| > 0$
    \begin{equation}\label{eq:TruncEquivKernel:LimAsDeltaToInf}
        \lim\limits_{\delta \to \infty} \varrho_{s,\delta,\text{eq}}(|\eta|) = \frac{C_{1,s}}{ |\eta|^{1+2s-2}} = \varrho_{s,\text{eq}}(|\eta|)\,.
    \end{equation}
\end{theorem}

\begin{proof}
    The smoothness and compact support of $\varrho_{s,\delta,\text{eq}}$ is apparent from the definition, and \eqref{eq:TruncKernel:LimAt1} follows easily from \eqref{eq:LimOfGAt1}.
    To see \eqref{eq:TruncEquivKernel:Nonnegative}, we recall that $F_s'(x) = \frac{1-x}{|1-x|^{2+s}} \cdot \frac{x}{|x|^{2+s}}$, so therefore $F_s(x)$ is increasing for $x \in (0,1)$, and thus $\varrho_{s,\delta,\text{eq}}(|\eta|) \geq 0$ for $\delta < |\eta| < 2 \delta$.
    Next, for $t \in (0,\delta)$
    \begin{equation*}
        \frac{d}{dt} \Big( F_s \Big( \frac{\delta}{t} \Big) - F_s \Big( 1-\frac{\delta}{t} \Big) - \kappa_s \Big)
        = 2 \frac{1-\frac{\delta}{t}}{|1-\frac{\delta}{t}|^{1+s}} \cdot \frac{\frac{\delta}{t} }{ |\frac{\delta}{t}|^{2+s} } \cdot \Big( \frac{-\delta}{t^2} \Big) > 0\,.
    \end{equation*}
    To see that $\varrho_{s,\delta,\text{eq}}(|\eta|) \geq 0$ for $0 < |\eta| < \delta$ it suffices to show that
    \begin{equation}\label{eq:TruncKernel:LimAt0}
        \lim\limits_{|\eta| \to 0} |\eta|^{1+2s-2} \varrho_{s,\delta,\text{eq}}(|\eta|) = (c_{1,s})^2 \left( \frac{2 \Gamma(1-s) \Gamma(1+2s)}{s \Gamma(1+s)} - \kappa_s \right) = C_{1,s} \,,
    \end{equation}
    where $C_{1,s} = \frac{2^{2s} s \Gamma(\frac{1}{2}+s) }{\pi^{1/2} \Gamma(1-s) }$ was defined in \Cref{example:equivalence:FractionalKernel}. The first equality follows from \eqref{eq:LimOfGAtInfty}, and the second equality follows from well-known identities satisfied by the Gamma function; these calculations are in \Cref{apdx:HyperGeometricFxn}. Since $C_{1,s}$ is clearly a positive number, we have established \eqref{eq:TruncEquivKernel:Nonnegative}.
    
    Now we prove \eqref{eq:TruncEquivKernel:Integrable}. By a change of variables and by definition of the support of $G_s$,
\begin{equation*}
    \begin{split}
        \intdm{\bbR}{\big| \varrho_{\delta,s,\text{eq}}(|\eta|) \big| }{\eta} &= 2 (c_{1,s})^2 \int_0^{\infty} \left| G_s \left( \frac{\delta}{\eta} \right) \right| \eta^{2-2s} \, \frac{\rmd \eta}{\eta} \\
        &= 2 (c_{1,s})^2 \delta^{2-2s} \int_0^{\infty} \frac{| G_s(r) |}{r^{2-2s}}  \, \frac{\rmd r}{r} \\
        &= C(s,\delta) \int_{1/2}^{\infty} \frac{| G_s(r) |}{r^{3-2s}}\, \rmd r\,.
    \end{split}
\end{equation*}

Since $G_s$ is continuous, by \eqref{eq:LimOfGAt1} there exists a $\tau > 0$ small such that $|G_s(r)| \leq \frac{4}{s} |r-1|^{-s}$ for all $r \in (1-\tau,1+\tau)$. Therefore since $G_s$ is bounded
\begin{equation*}
\begin{split}
    \int_{1/2}^{\infty} \frac{| G_s(r) |}{r^{3-2s}}\, \rmd r &\leq C \int_{ (\frac{1}{2},\infty) \setminus (1-\tau,1+\tau) } \frac{1}{r^{3-2s}}\, \rmd r + C \int_{(1-\tau,1+\tau) } \frac{1}{|r-1|^{s}} \frac{1}{r^{3-2s}}\, \rmd r < \infty \,.
\end{split}
\end{equation*}
Thus \eqref{eq:TruncEquivKernel:Integrable} is proved.

Finally, \eqref{eq:TruncEquivKernel:LimAsDeltaToInf} follows from the definition \eqref{eq:Truncated:DefnOfEquivKernel}, \eqref{eq:LimOfGAtInfty}, and the second equality in \eqref{eq:TruncKernel:LimAt0}.
\end{proof}

The properties of $\varrho_{s,\delta,\text{eq}}$ just established allow us to conclude that the formula \eqref{eq:Diffusion} holds.

\end{example}


\begin{example}
The tempered fractional kernel $\varrho_{s,\text{temp}}(|\bseta|)$ satisfies the conditions of \Cref{lma:EquivalenceKernel:Singular}. Upper and lower bounds for $d=1$ are calculated in \cite{Olson2020CSRI}.
Furthermore, we can show the following equivalence of energy spaces.
\begin{theorem}
    For $s \in (0,1)$, $\alpha > 0$, there exists $C = C(d,s,\alpha)$ such that
    \begin{equation*}
    \begin{split}
        \frac{1}{C} & \iintdm{\bbR^d}{\bbR^d}{\rme^{-\alpha|\bx-\by|} \frac{|\bu(\bx)-\bu(\by)|^2}{|\bx-\by|^{d+2s}}}{\by}{\bx} \\
        &\leq \iintdm{\bbR^d}{\bbR^d}{ \varrho_{s,\text{temp},\text{eq}}(|\bx-\by|) \frac{|\bu(\bx)-\bu(\by)|^2}{|\bx-\by|^2}}{\by}{\bx} \\
        &\leq C \iintdm{\bbR^d}{\bbR^d}{\rme^{-\alpha|\bx-\by|} \frac{|\bu(\bx)-\bu(\by)|^2}{|\bx-\by|^{d+2s}}}{\by}{\bx}
    \end{split}
    \end{equation*}
    for every $\bu \in \scS(\bbR^d;\bbR^d)$.
\end{theorem}
The proof uses techniques that are outside the scope of this paper, and so it will be reported elsewhere.
\end{example}

\begin{example}
The kernel defined in terms of the characteristic function
\begin{equation*}
    \varrho_{\chi,\delta}(|\bseta|) := \frac{d}{ \omega_{d-1} \delta^d} \chi_{B({\bf 0},\delta)}(|\bseta|)
\end{equation*}
satisfies \eqref{eq:KernelFullIntegrability}, and so \eqref{eq:Diffusion} holds immediately. Moreover, when $d=1$ we can find the equivalence kernel explicitly. A straightforward calculation shows that
\begin{equation*}
    \varrho_{\chi,\delta,\text{eq}}(|\eta|) = 
        \begin{cases}
            \frac{2 |\eta|}{\delta^2} \log \left( \frac{ \frac{\delta}{|\eta|} }{ |1-\frac{\delta}{|\eta|}| } \right)\,, & 0<|\eta| < 2\delta\,, \\
            0\,, & |\eta| \geq 2\delta\,.
        \end{cases}
\end{equation*}
Thus, $\varrho_{\chi,\delta,\text{eq}}$ is a nonnegative, integrable function.
\end{example}


\section{Helmholtz Decomposition for Fractional Operators}\label{sec:helmholtz}

In this section we combine the vector calculus identities proved in Section \ref{sec:identities} and the characterization of the equivalence kernel proved in Section \ref{sec:eq-kernel} to obtain a weighted fractional Helmholtz decomposition in H\"older spaces. Thus, we restrict our attention to the case of the fractional kernel $\varrho_s$ and utilize the results for H\"older spaces in Section \ref{sec:holder}.

First, we state the following result, whose proof can be obtained by using \cite[Theorem 2.8]{bucur2016some}.

\begin{theorem}\label{thm:FundSoln}
    Let $s \in (0,1)$ and $\sigma > 0$ be a sufficiently small quantity. Suppose $\bu \in \scC^{2s+\sigma}(\bbR^d;\bbR^d)$ with $d>2s$,
    and suppose $\bu$ is compactly supported. Define the constant 
    \begin{equation*}
        \kappa_{d,s} :=  \frac{ \Gamma(\frac{d}{2}-s) }{2^{2s} \pi^{d/2} \Gamma(s) }\,,
    \end{equation*}
    and define the function
    \begin{equation*}
        \Phi_s(\bsxi) :=
           \frac{\kappa_{d,s}}{|\bsxi|^{d-2s}}\,.
    \end{equation*}
    Then $\Phi_s$ is the fundamental solution of $(-\Delta)^s$ in the following sense: define the function
    \begin{equation*}
        \bv(\bx) := \Phi_s \ast \bu(\bx)\,, \qquad \bx \in \bbR^d\,.
    \end{equation*}
    Then $\bv$ belongs to $\scC^{2s+\sigma}(\bbR^d;\bbR^d)$,
    $\bv$ has the ``behavior at infinity"
    \begin{equation*}
        \Vnorm{\bv}_{L^1_{2s}(\bbR^d)} = \intdm{\bbR^d}{ \frac{|\bv(\bx)|}{ 1+|\bx|^{d+2s} } }{\bx} < \infty\,,
    \end{equation*}
    and both in the distributional sense and pointwise in $\bbR^d$ 
    \begin{equation*}
        (-\Delta)^s \bv(\bx) = \bu(\bx)\,.
    \end{equation*}
\end{theorem}



We can now state the main theorem of this section.

\begin{theorem}\label{thm:Helmholtz:Potentials}
    Let $0<s<1$. Suppose that $\bu \in \scC^{2s+\sigma}(\bbR^d;\bbR^d)$ with $d=3$ for some $\sigma > 0$ be sufficiently small.
    Suppose also that $\bu$ is compactly supported with $\supp \bu \subset B({\bf 0},R)$ for some $R >0$. Then there exist functions $\psi$ and $\bw$ belonging to $L^1_s(\bbR^d) \cap C^{0,s+\sigma}(\bbR^d)$ and $L^1_s(\bbR^d;\bbR^d) \cap C^{0,s+\sigma}(\bbR^d;\bbR^d)$ respectively such that
    \begin{equation}\label{eq:Helmholtz:Potentials}
        \bu(\bx) = \cG_s \psi(\bx) - \cC_s \bw(\bx) \quad \text{ for all } \bx \in \bbR^d\,.
    \end{equation}
\end{theorem}

\begin{proof}
By \Cref{thm:FundSoln}
\begin{equation*}
    \bu(\bx) = (-\Delta)^s \left[ \Phi_s \ast \bu(\bx) \right]\,.
\end{equation*}
Note that $\bu \in L^1_{2s}(\bbR^d;\bbR^d)$ since $\bu$ is continuous with compact support. By \Cref{prop:CurlOfCurl:SmoothFxns} and \Cref{prop:FractionalLaplaceIsDivGrad} we then have
\begin{equation}\label{Helmholtz:Potentials:Proof1}
    \bu(\bx) = \cG_s \circ \cD_s \left[ \Phi_s \ast \bu(\bx) \right] - \cC_s \circ \cC_s \left[ \Phi_s \ast \bu(\bx) \right]\,.
\end{equation} 

Define
\begin{equation}\label{Helmholtz:Potentials:Proof2}
    \begin{split}
        \psi(\bx) &:= \cD_s [ \Phi_s \ast \bu ] (\bx)\,, \\
        \bw(\bx) &:= \cC_s [ \Phi_s \ast \bu ] (\bx)\,.
    \end{split}
\end{equation}
Thus the formula \eqref{eq:Helmholtz:Potentials} will be established if we can show that $\psi$ and $\bw$ are well-defined functions.

To this end, note that both $\psi$ and $\bw$ are of the form $\cZ_s [\Phi_s \ast \bu]$.  These functions belong to $C^{0,s+\sigma}(\bbR^d)$ by \Cref{thm:MappingPropertiesOfOperators} since $\Phi_s \ast \bu \in \scC^{2s+\sigma}(\bbR^d)$ by \Cref{thm:FundSoln}.
Second, both $\psi$ and $\bw$ also belong to $L^1_{s}(\bbR^d;\bbR^d)$, being in $L^{\infty}(\bbR^d)$.
%
%
%
%
Thus by Theorem \ref{thm:MappingPropertiesOfOperators}, 1) the functions $\psi$ and $\bw$ are well-defined.
\end{proof}


\section*{Acknowledgments}	

M. D'Elia and M. Gulian are partially supported by the U.S. Department of Energy, Office of Advanced Scientific Computing Research under the Collaboratory on Mathematics and Physics-Informed Learning Machines for Multiscale and Multiphysics Problems (PhILMs) project (DE-SC0019453). 
They are also supported by Sandia National Laboratories (SNL). SNL is a multimission laboratory managed and operated by National Technology and Engineering Solutions of Sandia, LLC., a wholly owned subsidiary of Honeywell International, Inc., for the U.S. Department of Energy's National Nuclear Security Administration under contract {DE-NA0003525}. This paper, SAND2021-15379, describes objective technical results and analysis. Any subjective views or opinions that might be expressed in this paper do not necessarily represent the views of the U.S. Department of Energy or the United States Government.

T. Mengesha's research is supported by the NSF DMS 1910180.    

\appendix

\section{The Hypergeometric Function and Related Functions}\label{apdx:HyperGeometricFxn}

The power series defining the hypergeometric function $_2 F_1(a,b;c;z)$ converges for real parameters $a$, $b$, $c$ and complex $z$ in the unit disc -- except for the point $z = 1$ -- if $c-a-b \in (-1,0]$, which is the only range for the parameters that we are concerned about in this work. Its analytic continuation also exists everywhere except $1$ for $c-a-b \in (-1,0]$.

Using the identities \cite[Equations 15.2.4 and 15.1.8]{abramowitz1988handbook}, we get
\begin{gather}
    \label{eq:HypergeoDerivFormula1}
    \frac{d}{d z} \big( z^a \text{}_2 F_1 (a,b;a+1;z) \big) = a z^{a-1} \text{}_2 F_1 (a,b;a;z) = a z^{a-1} \text{}_2 F_1 (b,a;a;z)  = \frac{a z^{a-1} }{(1-z)^b}\,, \\
    \text{ or } \quad  \frac{d}{d z} \, \text{}_2 F_1 (a,b;a+1;z) = \frac{a}{z (1-z)^b} - \frac{a}{z} \,\text{}_2 F_1 (a,b;a+1;z)\,. \label{eq:HypergeoDerivFormula2}
\end{gather}

\begin{proof}[proof of \Cref{thm:PVofFs}]
To see \eqref{eq:AntiDer:Limitat0} when $s \neq 1/2$, use the definition of $F_s$ along with  \cite[Equation 15.3.6]{abramowitz1988handbook}:
\begin{equation*}
    \begin{split}
        F_s(\veps) - F_s(-\veps) &= \frac{1}{s(1-\veps)^s} \text{}_2 F_1 (-s,1+s;1-s;1-\veps) + \frac{1}{s \veps^s} \text{}_2 F_1 (-s,1+s;1-s;-\veps) \\
        &= \frac{1}{s(1-\veps)^s} \bigg( \frac{\Gamma(1-s) \Gamma(-s)}{\Gamma(-2s)} \text{}_2 F_1 (-s,1+s;1+s;\veps) \\
            &\qquad + \frac{1}{\veps^s} \frac{\Gamma(1-s) \Gamma(s)}{\Gamma(-s) \Gamma(1+s)} \text{}_2 F_1 (1,-2s;1-s;\veps)  \bigg) \\
            &\qquad + \frac{1}{s \veps^s} \text{}_2 F_1 (-s,1+s;1-s;-\veps)\,.
    \end{split}
\end{equation*}
Using $\Gamma(x+1) = x \Gamma(x)$, we see that $\frac{\Gamma(1-s) \Gamma(s)}{\Gamma(-s) \Gamma(1+s)} = -1$. Therefore,
\begin{equation*}
    \begin{split}
        F_s(\veps) - F_s(-\veps) &=  \frac{\Gamma(1-s) \Gamma(-s)}{s\Gamma(-2s)} \frac{1}{(1-\veps)^s} \text{}_2 F_1 (-s,1+s;1+s;\veps) \\
            &\qquad + \frac{1}{s \veps^s} \bigg( \text{}_2 F_1 (-s,1+s;1-s;-\veps) - \frac{1}{(1-\veps)^s}  \text{}_2 F_1 (1,-2s;1-s;\veps) \bigg)\,.
    \end{split}
\end{equation*}
By \cite[Equation 15.1.8]{abramowitz1988handbook} $\text{}_2 F_1 (-s,1+s;1+s;\veps) = (1-\veps)^s$, and by using the series definition of $\text{}_2 F_1$ we get
\begin{equation*}
    \begin{split}
        F_s(\veps) - F_s(-\veps) &=  \frac{\Gamma(1-s) \Gamma(-s)}{s\Gamma(-2s)} + \frac{1}{s \veps^s} \bigg( 1 - \frac{1}{(1-\veps)^s} \bigg) + O(\veps^{1-s})\,.
    \end{split}
\end{equation*}
Now, by \cite[Equation 15.1.8]{abramowitz1988handbook}
\begin{equation*}
    \begin{split}
        \frac{1}{\veps^s} \frac{(1-\veps)^s - 1}{(1-\veps)^s} = \frac{1}{\veps^s} \frac{\text{}_2 F_1 (-s,1;1;\veps) -  \text{}_2 F_1 (-s,1;1;0)}{(1-\veps)^s} = O(\veps^{1-s})\,.
    \end{split}
\end{equation*}
Therefore \eqref{eq:AntiDer:Limitat0} is proved in the case $s \neq 1/2$. When $s = 1/2$ we use the identities \cite[Equations 15.3.3 and 15.4.1]{abramowitz1988handbook} to explicitly compute the closed form of $\text{}_2 F_1$:
\begin{equation}\label{eq:HypergeometricCalc}
    \text{}_2 F_1 \left( -\frac{1}{2}, \frac{3}{2}; \frac{1}{2}; z\right) = \frac{1}{(1-z)^{1/2}}  \text{}_2 F_1 \left( -1, 1; \frac{1}{2}; z\right) = \frac{1-2z}{(1-z)^{1/2}}\,.
\end{equation}
Therefore
\begin{equation*}
    F_{\frac{1}{2}}(\veps) - F_{\frac{1}{2}}(-\veps) = \frac{2}{(1-\veps)^{1/2}} \frac{2\veps - 1}{\veps^{1/2}} + \frac{2}{\veps^{1/2}} \frac{1+2 \veps}{(1-\veps)^{1/2}} = O(\veps^{1/2})\,.
\end{equation*}

The proof of \eqref{eq:AntiDer:Limitat1} follows easily from the power series definition:
\begin{equation*}
    \begin{split}
        F_s(1+\veps) - F_s(1-\veps) = \frac{1}{s \veps^s} \text{}_2 F_1(-s,1+s;1-s;-\veps) - \frac{1}{s \veps^s} \text{}_2 F_1(-s,1+s;1-s;\veps) = O(\veps^{1-s})\,.
    \end{split}
\end{equation*}
\end{proof}

\begin{proof}[Proof of \Cref{thm:LimitsOfGs}]
    By definition $G_s$ is $C^{\infty}$ on $(0,\infty) \setminus (1-\tau,1+\tau)$. Then boundedness will follows from \eqref{eq:LimOfGAtInfty}. 
    
    To prove \eqref{eq:LimOfGAtInfty} when $s \neq 1/2$, we use the identity \cite[Equation 15.3.7]{abramowitz1988handbook} along with the value $\text{}_2 F_1(a,b;c;0) = 1$ for any $a$, $b$ and $c$:
\begin{equation*}
    \begin{split}
        \lim\limits_{x \to \infty} & G_s(x) \\
        &= \lim\limits_{x \to \infty} \frac{2}{s (x-1)^s } \text{}_2 F_1(-s,1+s;1-s;1-x) - \kappa_s \\
        &= \lim\limits_{x \to \infty} \frac{2}{s ( x-1 )^s } 
        \Bigg( \frac{\Gamma(1-s) \Gamma(1+2s)}{\Gamma(1+s) } \frac{1}{(x-1)^{-s}} \text{}_2 F_1 \left( -s,0; -2s ;\frac{1}{1-x} \right) \\
            &\qquad - \frac{\Gamma(1-s) \Gamma(-1-2s)}{\Gamma(-s) \Gamma(-2s) } \frac{1}{(x -1)^{1+s} } \text{}_2 F_1 \left( 1+s,1+2s;2+2s;\frac{1}{1-x} \right) \Bigg) - \kappa_s \\
        &= \lim\limits_{x \to \infty} \frac{2}{s} 
        \Bigg( \frac{\Gamma(1-s) \Gamma(1+2s)}{\Gamma(1+s) } \text{}_2 F_1 \left( -s,0; -2s ;0 \right) \\
            &\qquad - \frac{\Gamma(1-s) \Gamma(-1-2s)}{\Gamma(-s) \Gamma(-2s) } \frac{1}{(x-1)^{1+2s} } \text{}_2 F_1 \left( 1+s,1+2s;2+2s;0 \right) \Bigg) - \kappa_s \\
        &= \frac{2}{s} 
        \Bigg( \frac{\Gamma(1-s) \Gamma(1+2s)}{\Gamma(1+s) } \Bigg) - \kappa_s\,.
    \end{split}
\end{equation*}
To obtain \eqref{eq:LimOfGAtInfty} when $s = 1/2$ we use the the definition of the equivalence kernel for $x > 1$ and the identity \eqref{eq:HypergeometricCalc}:
\begin{equation*}
    \begin{split}
       G_{\frac{1}{2}}(x) &= \frac{4}{ (x-1)^{1/2} } \text{}_2 F_1 \left( -\frac{1}{2}, \frac{3}{2}; \frac{1}{2}; 1-x \right) - 0\\
        &=\frac{4}{ (x - 1)^{1/2} } \frac{\left( 1 - 2 \left( 1-x \right) \right)}{x^{1/2} }\,.
    \end{split}
\end{equation*}
Thus
\begin{equation*}
\begin{split}
    \lim\limits_{x \to \infty} G_{\frac{1}{2}}(x)
    = \lim\limits_{x \to \infty} 4 \left(  \left( 1 - x \right)^{1/2} + \frac{1}{ \left( 1 - x \right)^{1/2} }  \right)
    = 8 = \frac{2 \Gamma(\frac{1}{2}) \Gamma(2) }{\frac{1}{2} \Gamma(\frac{3}{2})}\,,
\end{split}
\end{equation*}
as desired.
The limit \eqref{eq:LimOfGAt1} follows from the left-hand and right-hand limits. First, since $\text{}_2 F_1 (a,b;c;0) = 1$
\begin{equation*}
    \begin{split}
        \lim\limits_{x \to 1^+ } |x-1|^s G_s(x) 
        &= \lim\limits_{x \to 1^+ } \frac{2}{s} \,\,\text{}_2 F_1 \left( -s,1+s;1-s;1-x \right) - \kappa_s  \left| x-1 \right|^{s} \\
        &= \frac{2}{s} \,\,\text{}_2 F_1 \left( -s,1+s;1-s;0 \right) = \frac{2}{s}\,.
    \end{split}
\end{equation*}
Second, 
\begin{equation*}
    \begin{split}
        \lim\limits_{x \to 1^- } |x-1|^s G_s(x)
        &= \lim\limits_{x \to 1^- } \frac{1}{s} \Bigg( \text{}_2 F_1 \left( -s,1+s;1-s;1-x \right) 
          - \frac{(1-x)^s}{x^s} \text{}_2 F_1 \left( -s,1+s;1-s;x \right)  \Bigg) \\
        &=  \frac{1}{s}  - \frac{1}{s} \lim\limits_{x \to 1^- } \frac{(1-x)^s}{x^s} \text{}_2 F_1 \left( -s,1+s;1-s;x \right)\,.
    \end{split}
\end{equation*}
We use the following limit for the hypergeometric function $\text{}_2 F_1 (a,b;c;z)$ to see the limit of this second expression: Using the transformation \cite[Equation 15.3.6]{abramowitz1988handbook} along with the the fact that $\text{}_2 F_1 (a,b;c;0) = 1$, we get
\begin{equation*}
    \lim\limits_{z \to 1^-} \frac{\text{}_2 F_1(a,b;c;z) }{(1-z)^{c-a-b}} = \frac{\Gamma(c) \Gamma(a+b-c)}{\Gamma(a) \Gamma(b)} \quad \text{ if } c-a-b < 0\,.
\end{equation*}
Therefore, since $(1-s)-(-s)-(1+s)=-s<0$
\begin{equation*}
\begin{split}
    \frac{1}{s}  - \frac{1}{s} \lim\limits_{x \to 1^- } \frac{(1-x)^s}{x^s} \text{}_2 F_1 \left( -s,1+s;1-s;x \right)
    = \frac{1}{s}  - \frac{1}{s} \frac{\Gamma(1-s) \Gamma(s)}{\Gamma(-s) \Gamma(1+s)} = \frac{2}{s}\,.
\end{split}
\end{equation*}
In the last equality we used that $\Gamma(z+1) = z \Gamma(z)$ for $z \notin \{0,-1,-2,\ldots \}$
\end{proof}

\begin{proof}[Proof of the second equality in \eqref{eq:TruncKernel:LimAt0}]
    First we assume $s \neq 1/2$ so that the relevant identities for the Gamma function are valid. Using the identity $\Gamma(x+1) = x \Gamma(x)$, 
    \begin{equation*}
    \begin{split}
         &(c_{1,s})^2 \left( \frac{2 \Gamma(1-s) \Gamma(1+2s)}{s \Gamma(1+s)} - \frac{\Gamma(1-s) \Gamma(-s)}{s \Gamma(-2s)} \right) \\
         &\quad = \frac{2^{2s}}{\pi} \left( \frac{\Gamma (1 + \frac{s}{2}) }{ \Gamma ( \frac{1-s}{2}) } \right)^2 \left( \frac{\Gamma(1-s)}{s} \right) \left( \frac{2 \Gamma(1+2s)}{\Gamma(1+s)} - \frac{ \Gamma(-s)}{\Gamma(-2s)}\right) \\
        &\quad = \frac{2^{2s}}{\pi}  \left( \frac{ \frac{s}{2} \Gamma ( \frac{s}{2}) }{ \Gamma ( \frac{1-s}{2}) } \right)^2 \left( \frac{\Gamma(1-s)}{s} \right) \left( \frac{4s \Gamma(2s)}{s\Gamma(s)} - \frac{ \Gamma(-s)}{\Gamma(-2s)}\right) \\
        &\quad = \frac{2^{2s} s }{4 \pi}  \left[ \Gamma \left(  \frac{s}{2} \right) \right]^2 \frac{1}{\Gamma(1-s)} \left( \frac{\Gamma(1-s)}{\Gamma ( \frac{1-s}{2} ) } \right)^2 \left( \frac{4 \Gamma(2s)}{\Gamma(s)} - \frac{ \Gamma(-s)}{\Gamma(-2s)}\right)\,.
    \end{split}
    \end{equation*}
    Now we use the Legendre duplication formula for the Gamma function:
    \begin{equation*}
    \begin{split}
        &\frac{2^{2s} s }{4 \pi}  \left[ \Gamma \left(  \frac{s}{2} \right) \right]^2 \frac{1}{\Gamma(1-s)} \left( \frac{\Gamma(1-s)}{\Gamma ( \frac{1-s}{2} ) } \right)^2 \left( \frac{4 \Gamma(2s)}{\Gamma(s)} - \frac{ \Gamma(-s)}{\Gamma(-2s)}\right) \\
        &\quad = \frac{2^{2s} s }{4 \pi}  \left[ \Gamma \left(  \frac{s}{2} \right) \right] ^2 \frac{1}{\Gamma(1-s)} \left( \frac{\Gamma(1-\frac{s}{2})}{2^s \pi^{1/2} } \right)^2 \left( 4 \frac{\Gamma(s+\frac{1}{2})}{2^{1-2s} \pi^{1/2} } - \frac{ 2^{1+2s} \pi^{1/2} }{\Gamma(\frac{1}{2}-s )}\right) \\
        &\quad = \frac{2^{-1+2s} s \Gamma \big( s+\frac{1}{2} \big) }{\pi^{1/2} \Gamma(1-s) }  \left( \frac{\Gamma( \frac{s}{2} ) \Gamma(1-\frac{s}{2})}{ \pi } \right)^2 \left( 1 - \frac{ \pi }{\Gamma(\frac{1}{2}-s ) \Gamma ( s+\frac{1}{2} )  }\right) \\
        &\quad = \frac{C_{1,s}}{2} \left( \frac{\Gamma( \frac{s}{2} ) \Gamma(1-\frac{s}{2})}{ \pi } \right)^2 \left( 1 - \frac{ \pi }{\Gamma(\frac{1}{2}-s ) \Gamma ( s+\frac{1}{2} )  }\right)\,.
    \end{split}
    \end{equation*}
    Finally, by Euler's reflection formula for the Gamma function and by elementary trigonometric identities,
    \begin{multline}
        \frac{C_{1,s}}{2}  \left( \frac{\Gamma( \frac{s}{2} ) \Gamma(1-\frac{s}{2})}{ \pi } \right)^2 \left( 1 - \frac{ \pi }{\Gamma(\frac{1}{2}-s ) \Gamma ( s+\frac{1}{2} )  }\right) \\
        = \frac{C_{1,s}}{2} \frac{1}{\sin^2 (\frac{\pi s}{2})} \left( 1 - \sin \left( \frac{\pi}{2} + \pi s \right) \right) = C_{1,s}\,.
    \end{multline}
    When $s = 1/2$, we can compute both sides of the equality in \eqref{eq:TruncEquivKernel:LimAsDeltaToInf} explicitly using $\Gamma(x+1) = x \Gamma(x)$, $\Gamma(1/2) = \sqrt{\pi}$ and $\Gamma(3/2) = \frac{\sqrt{\pi}}{2}$:
    \begin{equation*}
    \begin{split}
        (c_{1,\frac{1}{2}})^2 \left( \frac{2 \Gamma(\frac{1}{2}) \Gamma(2)}{\frac{1}{2} \Gamma(\frac{3}{2})} - 0 \right) = \frac{2}{\pi} \left( \frac{\Gamma(\frac{5}{4})}{\Gamma(\frac{1}{4})} \right)^2 \cdot 8 = \frac{16}{\pi} \left( \frac{\frac{1}{4}\Gamma(\frac{1}{4})}{\Gamma(\frac{1}{4})} \right)^2 = \frac{1}{\pi}\,,
    \end{split}
    \end{equation*}
    while
    \begin{equation*}
        C_{1,\frac{1}{2}} = \frac{1}{\sqrt{\pi} \Gamma(\frac{1}{2}) } = \frac{1}{\pi}\,.
    \end{equation*}
\end{proof}

\bibliographystyle{plainnat}
\bibliography{bibliography}

\end{document}